\documentclass[a4paper,12pt]{amsart}
\usepackage{lmodern}
\usepackage[T1]{fontenc}
\usepackage[utf8]{inputenc}

\usepackage{tikz-cd}

\usepackage{enumerate}

\usepackage{geometry}
\numberwithin{equation}{section}

\usepackage{amssymb}
\usepackage{mathtools}

\newtheorem{theo}{Theorem}[section]
\newtheorem{prop}[theo]{Proposition}

\newtheorem{coro}[theo]{Corollary}

\newtheorem{conj}[theo]{Conjecture}
\newtheorem{lemm}[theo]{Lemma}

\newtheorem{scol}[theo]{Scholium}
\theoremstyle{remark}
\newtheorem{rema}[theo]{Remark}

\newcommand{\ab}{\mathrm{ab}}

\newcommand{\an}{\mathrm{an}}
\newcommand{\aut}{\mathrm{aut}}
\newcommand{\BdR}{\mathbf{B}_{\mathrm{dR}}}
\newcommand{\C}{\mathbf{C}}

\DeclareMathOperator{\codim}{\mathrm{codim}}

\newcommand{\cris}{\mathrm{cris}}

\newcommand{\cyc}{\mathrm{cyc}}

\newcommand{\dR}{\mathrm{dR}}
\DeclareMathOperator{\End}{\mathrm{End}}

\renewcommand{\epsilon}{\varepsilon}
\DeclareMathOperator{\Fil}{\mathrm{Fil}}

\newcommand{\Fp}{\mathbf{F}_p}

\DeclareMathOperator{\Gal}{\mathrm{Gal}}
\renewcommand{\geq}{\geqslant}
\DeclareMathOperator{\GL}{\mathrm{GL}}
\DeclareMathOperator{\gr}{\mathrm{gr}}
\newcommand{\gtilde}{\widetilde{\mathfrak{g}}}

\DeclareMathOperator{\Hom}{\mathrm{Hom}}
\newcommand{\Id}{\mathrm{Id}}

\DeclareMathOperator{\Ind}{\mathrm{Ind}}
\renewcommand{\leq}{\leqslant}
\DeclareMathOperator{\Lie}{\mathrm{Lie}}

\newcommand{\pdR}{\mathrm{pdR}}

\newcommand{\Q}{\mathbf{Q}}
\newcommand{\Qp}{\mathbf{Q}_p}
\newcommand{\R}{\mathbf{R}}
\DeclareMathOperator{\rec}{\mathrm{rec}}

\DeclareMathOperator{\Res}{\mathrm{Res}}
\newcommand{\rig}{\mathrm{rig}}
\DeclareMathOperator{\rk}{\mathrm{rk}}

\newcommand{\sm}{\mathrm{sm}}
\DeclareMathOperator{\Sp}{\mathrm{Sp}}
\DeclareMathOperator{\Spec}{\mathrm{Spec}}
\DeclareMathOperator{\Spf}{\mathrm{Spf}}
\newcommand{\st}{\mathrm{st}}

\DeclareMathOperator{\Tri}{\mathrm{Tri}}
\newcommand{\tri}{\mathrm{tri}}

\newcommand{\vtilde}{\tilde{v}}
\newcommand{\Z}{\mathbf{Z}}
\newcommand{\Zp}{\mathbf{Z}_p}

\newcommand*{\isoarrow}[1]{\arrow[#1,"\rotatebox{90}{\(\sim\)}"]}

\newcommand{\dbl}{{\mathchoice{\mbox{\rm [\hspace{-0.15em}[}}
                              {\mbox{\rm [\hspace{-0.15em}[}}
                              {\mbox{\scriptsize\rm [\hspace{-0.15em}[}}
                              {\mbox{\tiny\rm [\hspace{-0.15em}[}}}}
\newcommand{\dbr}{{\mathchoice{\mbox{\rm ]\hspace{-0.15em}]}}
                              {\mbox{\rm ]\hspace{-0.15em}]}}
                              {\mbox{\scriptsize\rm ]\hspace{-0.15em}]}}
                              {\mbox{\tiny\rm ]\hspace{-0.15em}]}}}}

\author[E.~Hellmann, C.M.~ Margerin, B.~Schraen]{Eugen Hellmann, Christophe M. Margerin, Benjamin Schraen}

\address{Eugen Hellmann\\
Mathematisches Institut\\
Universit\"at M\"unster\\
Einsteinstrasse 62\\
D-48149 M\"unster\\
Germany\\
e.hellmann@uni-muenster.de}

\address{Christophe M. Margerin, CMLS, \'Ecole Polytechnique, F-91128 Palaiseau, France}

\address{Benjamin Schraen, Laboratoire de Mathématiques d’Orsay, Univ. Paris-Sud, CNRS, Université Paris-Saclay, 91405 Orsay, France\\
benjamin.schraen@math.u-psud.fr}

\title[Density of automorphic points]{Density of automorphic points in deformation rings of polarized global Galois representations}

\begin{document}

\begin{abstract}
Conjecturally, the Galois representations that are attached to essentially selfdual regular algebraic cuspidal automorphic representations are Zariski-dense in a polarized Galois deformation ring. 
We prove new results in this direction in the context of automorphic forms on definite unitary groups over totally real fields. This generalizes the infinite fern argument of Gouvea-Mazur and Chenevier, and relies on the construction of non-classical $p$-adic automorphic forms, and the computation of the tangent space of the space of trianguline Galois representations. This boils down to a surprising statement about the linear envelope of intersections of Borel subalgebras.
\end{abstract}

\maketitle

\tableofcontents

\section{Introduction}

Let $F$ be a number field, and fix a positive integer $n\geq1$ and a prime number $p$. The goal of this paper is to study some properties of deformation spaces of continuous representations
\[ \overline{\rho}:\,\Gal(\overline{F}/F)\longrightarrow\GL_n(\mathbf{F})\]
where $\mathbf{F}$ is a finite extension of $\Fp$. Assume that $\overline{\rho}$ is absolutely irreducible and unramified outside a finite set of places $S$ containing places dividing $p$. Mazur proved in \cite{Mazurdeform} that there exists a universal deformation of $\overline{\rho}$ unramified outside of $S$, that is, for $F_S\subset\overline{F}$ the maximal algebraic extension of $F$ unramified outside of $S$, a complete local noetherian ring $R_{\overline{\rho},S}$ and a continuous representation
\[ \rho_S^{\mathrm{univ}}:\,\Gal(F_S/F)\longrightarrow\GL_n(R_{\overline{\rho},S})\]
pro-representing the functor of deformations of $\overline{\rho}$ unramified outside $S$. The generic fiber $\mathcal{X}_{\overline{\rho},S}$ of the formal scheme $\Spf R_{\overline{\rho},S}$ is a rigid analytic space over $W(\mathbf{F})[\tfrac{1}{p}]$ whose closed points can be canonically identified with liftings of $\overline{\rho}$ to finite extensions of the $p$-adic field $W(\mathbf{F})[\tfrac{1}{p}]$.

When $F$ is a totally real field or a CM field, it is known that we can attached to each regular algebraic cuspidal automorphic representation $\pi$ of $\GL_n(\mathbb{A}_F)$ an $n$-dimensional $p$-adic continuous representation
\[ \rho_{\pi}:\,\Gal(\overline{F}/F)\longrightarrow\GL_n(\overline{\Qp}).\]
This representation is characterized by some local compatibility with $\pi$ at almost all finite places of $F$. As a consequence it is unramified outside a finite number of places. A very natural problem with regard to the rigid analytic spaces $\mathcal{X}_{\overline{\rho},S}$ concerns the distribution of automorphic points in $\mathcal{X}_{\overline{\rho},S}$, that is points corresponding to regular algebraic cuspidal automorphic representations $\pi$ of $\GL_n(\mathbb{A}_F)$ such that $\overline{\rho}_\pi$ reduces to $\overline{\rho}$ modulo $p$.

Beyond the case $n=1$ which is a consequence of class field theory, the case $n=2$, $F=\Q$ and $\overline{\rho}$ attached to a modular form has been completely solved by works of Gouvea-Mazur (\cite{GouveaMazur}) and Böckle (\cite{Bockledense}). It follows from their results that the space $\mathcal{X}_{\overline{\rho},S}$ is equidimensional of dimension $3$ and that the automorphic points are Zariski-dense inside $\mathcal{X}_{\overline{\rho},S}$.

For general values of $n$, the case of polarized representations has been studied by Chenevier in the paper \cite{Chefougere}. Let $\epsilon:\,\Gal(\overline{F}/F)\rightarrow\Zp^\times$ be the cyclotomic character. Assume moreover that $F$ is a totally imaginary quadratic extension of a totally real number field $F^+$ and let $c$ be the non trivial element of $\Gal(F/F^+)$. We recall that an $n$-dimensional $p$-adic representation $\rho:\,\Gal(\overline{F}/F)\rightarrow\GL_n(\overline{\Qp})$ is \emph{polarized} if there exists an isomorphism
\[ \rho^\vee\circ c\simeq \rho\otimes\epsilon^{n-1}.\]
When the regular algebraic cuspidal automorphic representation $\pi$ of $\GL_n(\mathbb{A}_F)$ is \emph{conjugate self dual}, that is $\pi^{\vee,c}\simeq \pi$, the representation $\rho_\pi$ is polarized. Moreover the representation $\rho_\pi$ is crystalline at $p$ if and only if the representations $\pi_v$ are unramified for $v\mid p$. In this situation, we have the following conjecture of Chenevier (\cite[Conj.~1.15]{Chefougere}) :

\begin{conj}\label{conjdensity}
Assume that $\overline{\rho}$ is absolutely irreducible and that there exists a regular conjugate self dual algebraic cuspidal automorphic representation $\pi$ of $\GL_n(\mathbb{A}_E)$ such that $\overline{\rho_\pi}\simeq\overline{\rho}$, then the set of points of the form $\rho_{\pi'}$ for $\pi'$ a regular conjugate self dual algebraic cuspidal automorphic representation unramified at $p$ is Zariski dense in $\mathcal{X}_{\overline{\rho},S}$.
\end{conj}

When $n=3$, $F_v=\Qp$ for $v\mid p$ and the deformation functor of $\overline{\rho}$ is \emph{unobstructed}, this conjecture has been proven by Chenevier.

The main result of this paper is the following.

\begin{theo}\label{maintheo}
Assume $p>2$ and the following assumptions
\begin{itemize}
\item the extension $F/F^+$ is unramified;
\item $2 \mid [F^+:\Q]$ if $n\equiv 2$ $(mod.\ 4)$;
\item $S$ contains only places which are split in $F$ ;
\item the representation $\overline{\rho}$ is absolutely irreducible and the group $\overline{\rho}(\Gal(\overline{F}/F(\zeta_p))$ is {\em adequate} in the sense of \cite{ThorneAut}.
\end{itemize}
Assume moreover that there exists some regular conjugate self dual cuspidal automorphic representation $\pi$ which is unramified outside of $S\setminus S_p$ and such that $\overline{\rho_\pi}\simeq\overline{\rho}$. Then the Zariski closure of automorphic points in $\mathcal{X}_{\overline{\rho},S}$ is a union of irreducible components.
\end{theo}

In the paper \cite{Allen2}, Patrick Allen proved that, assuming standard automorphy lifting conjectures, it is true that all irreducible components of the space $\mathcal{X}_{\overline{\rho},S}$ contain some regular conjugate self dual cuspidal automorphic point. As such points are smooth by \cite{Allen1}, Theorem \ref{maintheo} reaches substantial new cases of Conjecture \ref{conjdensity} under the standard automorphy lifting conjectures.

Let us also mention that David Guiraud proved in \cite{Guiraud} that, when the weight of $\pi$ satisfies a strong condition of regularity, the set of places $\lambda$ of $F$ where the reduction of the $\lambda$-adic representation associated to $\rho_\pi$ is unobstructed has density one.

Following \cite{Chefougere}, our strategy to prove Theorem \ref{maintheo} is to use base change results to deduce this result from an analogous result concerning automorphic forms on some definite unitary group $G$.

For this definite unitary group we can rely on a well developed theory of families of $p$-adic automorphic forms, so called eigenvarieties: there exists a rigid analytic space called the eigenvariety $Y(U^p,\overline{\rho})$ parametrizing overconvergent $p$-adic eigenforms on $G$. This space is a generalization of the eigencurve of Coleman and Mazur, and was first introduced by Chenevier in the setting of definite unitary groups. The existence of a family of Galois representations on $Y(U^p,\overline{\rho})$ gives rise to a map $Y(U^p,\overline{\rho})\rightarrow\mathcal{X}_{\overline{\rho},S}$. The image of this map is the so called ``infinite fern''. 

The main idea is to consider the Zariski-closure $\mathcal{X}_{\bar \rho, S}^{\aut}\subset \mathcal{X}_{\overline{\rho},S}$ of all automorphic points and show that each of its irreducible components contains a smooth point $\rho$ such that there is an equality of tangent spaces
\[T_\rho\mathcal{X}_{\bar \rho, S}^{\aut}=T_\rho\mathcal{X}_{\overline{\rho},S}.\]
We are hence reduced to proving that the left hand side is large enough. 

It is standard to prove that automorphic points form a Zariski dense subset of the eigenvariety and hence the canonical map
\begin{equation}\label{needsurjectivity}\bigoplus_x T_x Y(U^p,\overline{\rho})\longrightarrow T_\rho\mathcal{X}_{\overline{\rho},S}\end{equation}
factors through the tangent space $T_\rho\mathcal{X}_{\bar \rho, S}^{\aut}$, where the direct sum is indexed by all the preimages $x\in Y(U^p,\overline{\rho})$ of $\rho$.
Hence it will essentially suffice to prove that $(\ref{needsurjectivity})$ is surjective.  

One of the main results of \cite{BHS3} is the precise determination of this index set. In \cite{Chefougere} it is shown that the map in question is surjective, if the restriction of $\rho$ to the local Galois groups at places dividing $p$ satisfies some genericity assumption: roughly, the representations should be crystalline and the Hodge filtration in general position with respect to all possible Frobenius stable flags.
The main problem is that in higher dimensions there is (for the time being) no way to guarantee that $\mathcal{X}_{\overline{\rho},S}$ contains any point satisfying this assumption. The point of our paper is the proof of the surjectivity of $(\ref{needsurjectivity})$ without this genericity assumption.

As in \cite{Chefougere} we do so by proving a similar surjectivity result for local avatars of the spaces $\mathcal{X}_{\overline{\rho},S}$ and $Y(U^p,\overline{\rho})$: the global deformation ring is replaced by a local deformation ring, and the eigenvariety is replaced by the so called space of trianguline Galois representations.  The key construction of \cite{BHS3} is a local model for the space of trianguline representations. This local model allows us to reduce the surjectivity of $(\ref{needsurjectivity})$ to a problem in linear algebra.  

The problem is to determine the linear envelope of the intersection of a Borel algebra $\mathfrak{b}$ in the Lie-algebra $\mathfrak{gl}_n$ with the Weyl group translates of a fixed Borel $\mathfrak{b}_0$. This statement seems to be a very nice and interesting statement in its own right:

\begin{theo}\label{enveloppeBorelsintro}
Let $n$ be a positive integer, ${\mathfrak S}_{n}$ the symmetric group of order $n$, and $\mathfrak{gl}_n$ the algebra of $n\times n$ matrices with entries in a fixed field $k$. Let $\GL_{n}(k)$ be the group of the non-singular elements in  $\mathfrak{gl}_n$ and ${\mathfrak b}_{0}\subset \mathfrak{gl}_n$ the Borel subalgebra of upper triangular matrices. For any element $g\in\GL_{n}(k)$ let ${\mathfrak b}_{g}=g^{-1} {\mathfrak b}_{0} g$ denote the Borel subalgebra conjugate to ${\mathfrak b}_{0}$ by $g^{-1}$.
  
\noindent
Any Borel subalgebra ${\mathfrak b}$ coincides with the linear envelope of its intersections with the conjugate of ${\mathfrak b}_{0}$ under ${\mathfrak S}_{n}$
\[
{\mathfrak b}  = \sum_{w \in {\mathfrak S}_{n}}{\mathfrak b} \cap {\mathfrak b}_{w}.
\]
\end{theo}

The plan of the paper is the following. In a first section, we prove Theorem \ref{enveloppeBorelsintro} and prove as an application a surjectivity result for a map between tangent spaces of our local models. The second section is purely local and its purpose is to prove our main local result concerning the sum of tangent spaces of quasi-trianguline deformation spaces. Finally the last section is of global nature and contains the proof of our main global theorem.

\noindent\emph{Notation} : We fix a prime number $p$. 
Let $K$ be a finite extension of $\Qp$, $\overline{K}$ an algebraic closure of $K$ and $\widehat{\overline{K}}$ its completion for the unique valuation extending the $p$-adic valuation. We denote by $v_K$ the unique valuation of $K$ taking the value $1$ on uniformizers of $K$. We use the notation $\mathcal{G}_K$ for the Galois group $\Gal(\overline{K}/K)$. 

We fix $L$ a finite extension of $\Qp$. Let $\Sigma$ be the set of $\Qp$-algebra homomorphisms from $K$ to $L$. We choose $L$ big enough so that $|\Sigma|=[K:\Qp]$, or equivalently $L\otimes_{\Qp}K\simeq L^{[K:\Qp]}$,
and we denote by $\mathcal{O}_L$ the ring of integers of $L$, by $\mathfrak{m}_L$ the unique maximal ideal of $\mathcal{O}_L$ and by $k_L\coloneqq \mathcal{O}_L/\mathfrak{m}_L$ its residue field. Let $|\cdot|_p$ be the unique norm on $L$ inducing the $p$-adic norm on $\Qp$. Let $\epsilon_L$ be the character $N_{L/\Qp}|N_{L/\Qp}|$  from $L^\times$ to $\Zp^\times$. Let $\rec_L:\, L^\times\rightarrow W_L^{\ab}$ be the local reciprocity isomorphism sending a uniformizer of $L$ onto a \emph{geometric} Frobenius element. We have $\epsilon_L=\chi_{\cyc}\circ\rec_L$.

If $X$ is some algebraic variety defined over $K$, we will use the notation $X_{K/\Qp}$ for the Weil restriction of $X$ from $K$ to $\Qp$. Moreover if $Y$ is some algebraic variety defined over $\Qp$, we will use the notation $Y_L$ for the base change of $Y$ from $\Qp$ to $L$, so that $X_{K/\Qp,L}=(X_{K/\Qp})\times_{\Spec \Qp}\Spec L$. As we have $L\otimes_{\Qp}K\simeq L^{\Sigma}$, we have an isomorphism of algebraic varieties over $L$
\begin{equation}\label{split} X_{K/\Qp,L}\simeq\prod_{\tau\in\Sigma} X_{\tau} \end{equation}
where $X_{\tau}$ is the base change of $X$ from $K$ to $L$ via the embedding $\tau$. If $x$ is some $L$-point of $X_{K/\Qp}$ we will denote $(x_{\tau})\in\prod_{\tau\in\Sigma}X_{\tau}$ its image by the isomorphism \eqref{split}.

If $\mathbf{k}\in(\Z^n)^{[K:\Qp]}$, we define the algebraic character $\delta_{\mathbf{k}}:\,T(K)\rightarrow L^{\times}$ by the formula
\[ (a_1,\dots,a_n)\mapsto \prod_{i=1}^n\prod_{\tau\in\Sigma}\tau(a_i)^{k_{i,\tau}}\]

If $X$ is a scheme, or a rigid analytic space and $x\in X$ is a point, we write $T_xX$ for the tangent space of $X$ at $x$. Similarly, if $\mathfrak{X}$ is a deformation functor defined on local Artin rings with fixed residue field (or a formal scheme pro-representing such a functor), we write $T\mathfrak{X}$ for the tangent space of $\mathfrak{X}$ at the unique closed point. 

\section{On intersections of Borel algebras}\label{intersection}

\subsection{Envelopes of intersections of Borel subalgebras}

Let $n$ be a positive integer and $k$ a field. We denote by $\mathfrak{gl}_n$ the Lie algebra of $n\times n$-matrices with coefficients in $k$ and by $\mathfrak{b}_0\subset \mathfrak{gl}_n$ the Borel subalgebra of upper triangular matrices. Given an element $g\in \GL_{n}(k)$ we write ${\mathfrak b}_{g}= g^{-1}{\mathfrak b}_{0} g$ for the conjugate of $\mathfrak{b}_0$ by $g$. We denote by $B$ the subgroup of upper triangular matrices in $\GL_{n}(k)$ and by $W$ the Weyl group $N/T$, where $N$ stands for the subgroup of matrices with exactly one non-zero entry in each row and each column and $T$ for the subgroup of diagonal matrices, $T= B\cap N$; the Weyl group $W$ identifies with the subgroup of $\GL_{n}(k)$ of $n \times n$ permutation matrices, and as such is isomorphic to the group of permutations over $n$ elements,  ${\mathfrak S}_{n}$. When speaking of elements of maximal length in $W$ we refer to the generating set $S$ of $W$ whose elements are (the permutation matrices associated with) the transpositions $(j,j+1)$, $j\in [[1,n-1]]$, so that the quadruple $(\GL_{n}(k),B,N,S)$ forms a Tits system (\cite[IV, 2.2]{BourbakiLie456}). Hereafter we freely identify an element $w$ of $W$ with its image in $\GL_{n}(k)$, the associated permutation matrix, and with its image in ${\mathfrak S}_{n}$, the underlying permutation. All scalars to be considered will be taken in $k$.

\noindent
By its very definition the linear envelope $\sum_{w \in W } {\mathfrak b} \cap {\mathfrak b}_{ w}$ of the intersection of any Borel subalgebra ${\mathfrak b}\subset \mathfrak{gl}_n$ with the conjugates of ${\mathfrak b}_{0}$ under the Weyl group, is contained in ${\mathfrak b}$; we discuss here the reverse inclusion and show the nice identity,
\[ {\mathfrak b} = \sum_{w \in W} {\mathfrak b} \cap {\mathfrak b}_{w},\]

\noindent
that states the envelope does coincide with ${\mathfrak b}$.

\noindent
Since any Borel subalgebra of $\mathfrak{gl}_n$ is a conjugate of the standard Borel subalgebra ${\mathfrak b}_{0}$, it will be enough to establish the identity for ${\mathfrak b} = {\mathfrak b}_{g}$, for an arbitrary element  $g \in \GL_{n}(k)$. 
By the Bruhat decomposition, every such element $g$ splits as a product $g=u_{1}s u_{2}$ of two (invertible) upper triangular matrices, $u_{1}$ and $u_{2}$ and a permutation matrix $s$ associated to a permutation $s \in {\mathfrak S}_{n}$, so that the identity to discuss reads
\[
{\mathfrak b}_{su} = \sum_{w \in W} {\mathfrak b}_{su} \cap {\mathfrak b}_{w},
\]
\noindent
which is to hold for an arbitrary $n\times n$ permutation matrix $s$ and an arbitrary upper triangular matrix $u \in B=\GL_{n}(k) \cap {\mathfrak b}_{0} $.

\vspace{1cm}
\par

In a first part we settle this identity for $w_{0}$ {\em the} permutation of maximal length in ${\mathfrak S}_{n}$, i.e. the involution ~$w_{0}=(1,n)(2,n-1)...(\lfloor{\frac{n}{2}}\rfloor,n-\lfloor\frac{n}{2}\rfloor +1)$. Since conjugation by any element $g \in \GL_{n}(k)$, is a linear isomorphism of $\mathfrak{gl}_n$, the conjugate of a linear envelope coincides with the envelope of the conjugates, and since intersection and conjugation trivially commute, we find \[\sum_{w \in W } {\mathfrak b}_{w_0 u} \cap {\mathfrak b}_{w} =  \sum_{w\in W} ({\mathfrak b}_{w_0} \cap {\mathfrak b}_{wu^{-1}})_{u}.\] Hence the envelope $\sum_{ w\in W } {\mathfrak b}_{w_0 u} \cap {\mathfrak b}_{w}$ coincides with the Borel subalgebra ${\mathfrak b}_{w_0 u}$ if and only if ${\mathfrak b}_{w_0} = \sum_{ w\in W} {\mathfrak b}_{w_0} \cap {\mathfrak b}_{wu^{-1}}$ (the reader will note the Borel subalgebra ${\mathfrak b}_{w_0}$ coincides with the Borel algebra of lower triangular matrices).

The proof proceeds through an explicit ``d\'evissage'', which the following lemma will make clear; It does {\em not} rely on any induction on the dimension, nor does it require any further assumption on the fixed base field $k$: anyone will do. The elementary $n\times n$ matrix whose $(l,m)$-entry is given by $\delta_{i,l}\delta_{j,m}$ for $l,m \in [\![1,n]\!]$ will be denoted by $e^{i,j}$ and by $x_{i,j}$, $i, j \in [\![1,n]\!]$, we will denote scalars in the base field $k$.

\begin{lemm}{\label{lem1}} Let $k$ be an arbitrary field. For any $u\in B$, and for any ordered pair $(i,j)$ in $[\![1,n]\!]^{2}$, $i \geq j$, there is some permutation $s_{i,j} \in {\mathfrak S}_{n}$, and $i-j$ scalars $(x_{i,l})_{l=j+1}^{i}$ in $k$, such that the matrix $ a^{i,j}\coloneqq e^{i,j} + \sum_{l=j+1}^{i}x_{i,l}e^{i,l}$ lies in the Borel subalgebra ${\mathfrak b}_{s_{i,j}u^{-1}}$.

  The matrices $a^{i,j}$ with $(i,j)$ in $[\![1,n]\!]^{2}$, $i \geq j$, then form a basis of the Borel subalgebra ${\mathfrak b}_{w_0}$ of lower triangular matrices in $\mathfrak{gl}_n$ all elements of which lie in the envelope $\sum_{w \in {\mathfrak S}_{n}} {\mathfrak b}_{w_0} \cap {\mathfrak b}_{w u^{-1}}$.
  
The permutation  $s_{i,j}$ may be chosen to be the $(i,j)$-transposition.
\end{lemm}

\begin{proof}[Proof of Lemma \ref{lem1}] Let's denote by $u_{i,j}$ the entries of the upper triangular matrix $u$, so that $u_{i,j} \in k$, $u_{i,j}=0$ if $i>j$, and $\prod_{i=1}^{n}u_{i,i}\neq 0$. From the properties of the inverse matrice $u^{-1}$ of $u$ we will only use it is some non-singular upper triangular matrix, which we'll denote by $v$.

\noindent
Because of obvious support properties of the involved matrices, the general entry of the product matrix $p\coloneqq va^{i,j}u$ vanishes \--- i.e. $p_{l,m}$=0 \--- for all $l, l \in [\![1,n]\!]$, as soon as $m \in [\![1,j-1]\!]$ and for all $m, m \in [\![1,n]\!]$, as soon as $l \in [\![i+1,n]\!]$: in words, the first $j-1$ columns and last $n-i$ rows of the product $p$ vanish identically. One may further observe all entries of the matrix $q\coloneqq a^{i,j}u$ vanish on the same ground, but for the entries $q_{i,l}$, $l \in [\![j,n]\!]$, which are given by the identities: $q_{i,l}= u_{j,l}+ \sum_{m=j+1}^{\min(k,i)}u_{m,l}x_{i,m}$. From these identities we can then easily deduce the following expression for the entries $p_{a,b}$, $a,b \in [\![j,i]\!]$:

\begin{equation}{\label{1}}
p_{a,b}= v_{a,i}(u_{j,b}+u_{j+1,b}x_{i,j+1}+...+u_{b,b}x_{i,b}) = v_{a,i}(u_{j,b} + \sum_{l=j+1}^{b}u_{l,b}x_{i,l}).
\end{equation}

\noindent
It can be further remarked that the matrix $r\coloneqq p_{s_{(i,j)}}$, obtained by swapping both the $i$-th and $j$-th columns and rows, \--- so that $p_{s_{(i,j)}}$ coincides with the conjugate of $p$ by the permutation matrix $s_{(i,j)}$ associated with the $(i,j)$-transposition, consistently with previously introduced notations \---, shares with $p$ the property that all entries below the main diagonal and outside the block $l\in [\![j,i]\!] $, $m \in [\![j,i]\!]$ \--- a $(i-j+1) \times (i-j+1)$-submatrix \---, vanish, to the effect that $r$ will lie in ${\mathfrak b}_{0}$, $p$ in ${\mathfrak b}_{s_{(i,j)}}$ and so $a^{i,j}$ in ${\mathfrak b}_{s_{(i,j)}u^{-1}}$ as soon as all sub-diagonal entries of this block in $r$ cancel.

\noindent
Identity (\ref{1}) above gives precisely the following expressions for the entries $r_{a,b}$ in the block: first, for all $b \in [\![j+1,i-1]\!]$ and all $a \in [\![j+1,i-1]\!]$,
\begin{align*}
r_{j,b}&=p_{i,b}= v_{i,i}\ (u_{j,b}+ \sum_{l=j+1}^{b}u_{l,b}x_{i,l},)\\
r_{a,b}&=p_{a,b}= v_{a,i}\ (u_{j,b}+ \sum_{l=j+1}^{b}u_{l,b}x_{i,l}),\ \text{and}\\
r_{i,b}&=p_{j,b}= v_{j,i}\ (u_{j,b}+ \sum_{l=j+1}^{b}u_{l,b}x_{i,l});
\end{align*}

\noindent
As for the first column of the block, which is to extract from the $j$-th column of $r$, itself coinciding with the $i$-th column of $p$ \--- up to swapping the $i$-th and $j$-th entries \---, we get

\begin{align*}
r_{j,j}&=p_{i,i}= v_{i,i}\ (u_{j,i}+ \sum_{l=j+1}^{i}u_{l,i}x_{i,l}),\\
r_{a,j}&=p_{a,i}= v_{a,i}\ (u_{j,i}+ \sum_{l=j+1}^{i}u_{l,i}x_{i,l}),\\
r_{i,j}&=p_{j,i}= v_{j,i}\ (u_{j,i}+ \sum_{l=j+1}^{i}u_{l,i}x_{i,l}).
\end{align*}

\noindent
At this point it should be clear the $b$-th column, $b \in [\![j+1,i-1]\!]$, vanishes identically as soon as the following affine relation is satisfied by the scalars $(x_{i,l})_{l=j+1}^{i}$

\begin{equation}{\label{2}}
u_{j,b}+ \sum_{l=j+1}^{b}u_{l,b}x_{i,l}=0  .
\end{equation}

From the second set of equations, it follows the $j$-th column will also cancel provided the following affine relation 
is satisfied by the scalars $(x_{i,l})_{l=j+1}^{i}$

\begin{equation}{\label{3}}
u_{j,i}+ \sum_{l=j+1}^{i}u_{l,i}x_{i,l}=0  .
\end{equation}

\noindent
It remains to organize the system of equations (\ref{2}) \--- for $b \in [\![j+1,i-1]\!]$ \---, and the equation (\ref{3}) into some lower triangular linear system of size $i-j$ as follows

\begin{equation}{\label{4}}
\begin{pmatrix} u_{j+1,j+1} & 0 & 0 & \hdots & 0 \\ u_{j+1,j+2} & u_{j+2,j+2} & 0 & \hdots & 0 \\  u_{j+1,j+3} & u_{j+2,j+3} & u_{j+3,j+3} & \hdots & 0 \\\vdots & \vdots & \vdots & \ddots & \vdots \\ u_{j+1,i} & u_{j+2,i} & u_{j+3,i} & \hdots & u_{i,i} \end{pmatrix}\quad \begin{pmatrix} x_{i,j+1}\\x_{i,j+2}\\x_{i,j+3}\\\vdots\\x_{i,i} \end{pmatrix} = - \begin{pmatrix} u_{j,j+1}\\u_{j,j+2}\\u_{j,j+3}\\\vdots\\u_{j,i} \end{pmatrix}.
\end{equation}

\noindent
Since the matrix $u$ is non-singular by assumption, the product $\prod_{l=j+1}^{i}u_{l,l}$ does not vanish and the previous system admits a (unique) solution; for such a solution ${\underline x}= (x_{i l})_{l=j+1}^{i}$ the matrix $a^{i,j}\coloneqq  a^{i,j}({\underline x})$, $i, j  \in [\![1,n]\!]$, $i \geq j$, lies in ${\mathfrak b}_{s_{(i,j)}u^{-1}} \cap {\mathfrak b}_{w_0}$, where as above we denote by $s_{(i,j)}$ the permutation matrix in $GL_{n}(k)$ associated with the $(i,j)$-transposition in ${\mathfrak S}_{n}$. This closes the proof of the first and main claim in Lemma \ref{lem1}.

By construction one then passes from the matrices $e^{i,j}$ to the matrices $a^{i,j}$, $i, j  \in [\![1,n]\!]$, $i \geq j$, by some unipotent triangular matrix (with many zeros, since it is $n$-block diagonal), provided the order we choose on the set $\{(i,j), i, j  \in [\![1,n]\!]$, $i \geq j \}$ is compatible with the row order, i.e. such that for all $i \in [\![1,n]\!] $ $(i,j)< (i,k) $ if $j<k$ (the lexicographic order clearly has the property); the matrices  $a^{i,j}$, $i, j  \in [\![1,n]\!]$, $i \geq j$, then form a basis of the Borel subalgebra ${\mathfrak b}_{w_0}$, while each one lying in exactly one of the generators $ {\mathfrak b}_{w_0} \cap {\mathfrak b}_{s_{(i,j)}u^{-1}}$ of the envelope $\sum_{w \in \mathfrak W} {\mathfrak b}_{w_0} \cap {\mathfrak b}_{wu^{-1}}$, the second and last claim in Lemma \ref{lem2}.
\end{proof}

The reader may want to notice the above argument reaches the stronger statement that for all $u$, $u \in B$, ${\mathfrak b}_{w_{0}u}=$ $\sum_{t \in {\mathfrak T}_{n}} {\mathfrak b}_{w_{0}u} \cap {\mathfrak b}_{t}$, where the sum is taken over the subset ${\mathfrak T}_{n} \subset W_{0}$ consisting of the identity and the $i,j$-transpositions, $n\geq i>j\geq 1$, a small subset of ${\mathfrak S}_{n}$ with only $(n^{2}-n+2)/2$ elements.

We now discuss the details of the reduction of the general statement to Lemma \ref{lem1}, although it may look less surprizing to trained minds.
Let's proceed to the reduction. The following lemma is a direct consequence of a refined version of the Bruhat decomposition established in \cite[Thm.~5.15]{BorelTits}.

\begin{lemm}{\label{lem2}}
  Let $k$ be an arbitrary field. Any $n\times n$-matrix in $\mathfrak{gl}_n$ splits as the product of an upper triangular matrix by a lower triangular matrix by a permutation matrix: for all matrices $m \in \mathfrak{gl}_n$, there exist an upper triangular matrix $u$, a lower triangular matrix $l$, and a permutation matrix $p$, such that $m=ulp$.

If useful, one may require the upper triangular $u$  or the lower triangular $l$ to be unipotent (but not both simultaneously of course). 
  
\end{lemm}

Let's turn back to our main object, realizing a Borel subalgebra as the envelope of its intersections with the conjugates of any fixed Borel subalgebra under the Weyl group. From Lemmas \ref{lem1} and \ref{lem2} we can deduce the following statement.

\begin{theo}\label{enveloppeBorels}
Let $n$ be a positive integer, ${\mathfrak S}_{n}$ the symmetric group of order $n$, and $\mathfrak{gl}_n$ the algebra of $n\times n$ matrices with entries in a fixed field $k$. Let $\GL_{n}(k)$ be the group of the non-singular elements in  $\mathfrak{gl}_n$ and ${\mathfrak b}_{0}\subset \mathfrak{gl}_n$ the Borel subalgebra of upper triangular matrices. For any element $g\in\GL_{n}(k)$ let ${\mathfrak b}_{g}=g^{-1} {\mathfrak b}_{0} g$ denote the Borel subalgebra conjugate to ${\mathfrak b}_{0}$ by $g^{-1}$.
  
\noindent
Any Borel subalgebra ${\mathfrak b}$ coincides with the linear envelope of its intersections with the conjugate of ${\mathfrak b}_{0}$ under ${\mathfrak S}_{n}$
\[
{\mathfrak b}  = \sum_{w \in {\mathfrak S}_{n}}{\mathfrak b} \cap {\mathfrak b}_{w}.
\]
\end{theo}

\begin{proof} All Borel subalgebras are known to be conjugate, and it is enough to prove the identity in the theorem for ${\mathfrak b} = {\mathfrak b}_{g}=g^{-1}{\mathfrak b}_{0}g$ for all $g \in GL_{n}(k)$. By Lemma \ref{lem2} the element $g$ splits as a product of an upper triangular matrix, $u$, by a lower triangular matrix, $l$, by a permutation matrix, $p$,i.e.~we can write $g=ulp$. The matrix $l$ can be written as the conjugate $l=w_0 u_{2}w_0^{-1}$ of an upper triangular matrix $u_{2}$ by the permutation matrix $w_0$ associated with the permutation of maximal length in  ${\mathfrak S}_{n}$, that is \[w_{0}=(1,n)(2,n-1)\hdots(\lfloor{\frac{n}{2}}\rfloor,n - \lfloor\frac{n}{2}\rfloor +1).\] Substituting for $l$ in $g=ulp$ accordingly, and introducing the permutation matrix $q\coloneqq w_0^{-1}p$ we get $g=uw_0 u_{2}w_0^{-1}p=u w_0 u_{2} q$, and for the Borel subalgebra ${\mathfrak b}_{g}= {\mathfrak b}_{w_0 u_{2} q}$.

\noindent
Now, as observed above, conjugation trivially commutes with taking linear envelope and intersection so that $\sum_{w \in {\mathfrak S}_{n}} {\mathfrak b}_{w_0 u_{2}} \cap {\mathfrak b}_{w}$ coincides with $(\sum_{w \in {\mathfrak S}_{n}} {\mathfrak b}_{w_0} \cap {\mathfrak b}_{w u_{2}^{-1}})_{u_{2}}$ and the identity ${\mathfrak b}_{w_0 u_{2}}= \sum_{w \in {\mathfrak S}_{n}} {\mathfrak b}_{w_0 u_{2}} \cap {\mathfrak b}_{w} $ is equivalent to ${\mathfrak b}_{w_0} = \sum_{w \in {\mathfrak S}_{n}} {\mathfrak b}_{w_0} \cap {\mathfrak b}_{w u_{2}^{-1}}$, which, in turn, is precisely the conclusion of Lemma \ref{lem1}.

\noindent
Again, conjugation commutes with taking linear envelope and intersection, to the effect that the identity ${\mathfrak b}_{w_0 u_{2}}= \sum_{w \in {\mathfrak S}_{n}} {\mathfrak b}_{w_0 u_{2}} \cap {\mathfrak b}_{w} $ reads
\[
{\mathfrak b}_{g} = {\mathfrak b}_{w_0 u_{2} q}= (\sum_{w \in {\mathfrak S}_{n}} {\mathfrak b}_{w_0 u_{2}} \cap {\mathfrak b}_{w})_{q} = \sum_{w \in {\mathfrak S}_{n}} ({\mathfrak b}_{w_0 u_{2}} \cap {\mathfrak b}_{w})_{q} = \sum_{w \in {\mathfrak S}_{n}} {\mathfrak b}_{w_0 u_{2} q} \cap {\mathfrak b}_{w q} = \sum_{w \in {\mathfrak S}_{n}} {\mathfrak b}_{g} \cap {\mathfrak b}_{w q}.
\]
\noindent
Since in any group (right-) translation by any element is a bijection, the latter sum can be rewritten $\sum_{w \in {\mathfrak S}_{n}} {\mathfrak b}_{g} \cap {\mathfrak b}_{w}$, which reaches the claim that

\begin{equation*}
{\mathfrak b}_{g} = \sum_{w \in {\mathfrak S}_{n}}{\mathfrak b}_{g}\cap {\mathfrak b}_{w}. 
\end{equation*}

\noindent
and closes the proof of Aequatio Praeclara.
\end{proof}

One will note the remark closing the proof of Lemma {\ref{lem1}} that conjugations under only a small subset of ${\mathfrak S}_{n}$ are needed to recover ${\mathfrak b}_{w_{0}u}$ implies, when introduced in the previous discussion, that in general it is enough to consider the envelope of the intersections of the Borel subalgebra ${\mathfrak b}_{g}$ with the Borel subalgebras ${\mathfrak b}_{t q}$, $t \in {\mathfrak T}_{n}$, where $q = w_{0}^{-1}p$, for $p$ the permutation factor of the $ulp$ decomposition of $g$: the sum in the above decomposition can be taken on a prescribed translate of ${\mathfrak T}_{n}$, a small subset (of cardinal $(n^{2} - n +2)/2)$ of ${\mathfrak S}_{n}$.

\subsection{A surjectivity result}\label{tangentlocalmodel}

Let $k$ be a field and let $G\coloneqq\GL_{n,k}$. Let $B$ be a Borel subgroup in $G$ and let $\mathfrak{g}\coloneqq\Lie G$ and $\mathfrak{b}\coloneqq\Lie B$. The quotient scheme $G/B$ is identified to the projective scheme classifying the complete flag in $k^n$. As two complete flags in $k^n$ are conjugate under $G(k)$, we have a natural isomorphism of sets $(G/B)(k)\simeq G(k)/B(k)$, so that we will identify $k$-points of $G/B$ with right cosets $gB(k)$, for $g\in G(k)$. Let $\gtilde$ be the Grothendieck simultaneous resolution of $\mathfrak{g}$: it coincides with the closed subscheme $\{(A,gB)\mid \mathrm{Ad}(g^{-1})A\in \mathfrak{b}\}$ of the product $\mathfrak{g}\times_k G/B$. Let $\pi_1$ and $\pi_2$ be the projections of $\gtilde\times_{\mathfrak{g}}\gtilde$ onto $\gtilde$ with respect to the first and second factors.

\begin{lemm}\label{tangent1}
Let $g\in G(k)$. The tangent space of $\gtilde$ at the point $(0,gB(k))$ of $\mathfrak{g}\times G/B$ is the $k$-linear subspace $g\mathfrak{b}g^{-1}\oplus T_{gB(k)}G/B$ of $\mathfrak{g}\oplus T_{gB(k)}G/B$. 
\end{lemm}

\begin{proof}
Let $B^-$ be the Borel subgroup of $G$, opposite to $B$ and let $U^-$ be the unipotent radical of $B^-$. There is an open embedding of $U^-$ into $G/B$ sending $u\in U^-$ to $guB$ (see for example \cite[II.1.10]{Jantzen}). This open embedding induces an isomorphism $T_{gB(k)} G/B\simeq T_{\Id} U^-$. This implies that the tangent space $T_{(0,gB(k))}\gtilde$ can be identified with the set of pairs $(A, C)$ in $\mathfrak{g}\times \Lie U^-$ such that $(\varepsilon A, g(\Id+\varepsilon C))\in\gtilde(L[\varepsilon])$, which means
\begin{equation}\label{tangentrelation}(g(\Id+\varepsilon B))^{-1}\varepsilon A g(\Id+\varepsilon C)\in k[\varepsilon]\otimes_k\ \mathfrak{b}.
\end{equation}
Using the fact that $\varepsilon^2=0$, \eqref{tangentrelation} is equivalent to $g^{-1} A g\in\mathfrak{b}$.
\end{proof}

The same kind of computation shows the following result.

\begin{lemm}\label{tangent2}
The tangent space of $\gtilde\times_{\mathfrak{g}}\gtilde$ at the point $(gB(k),0,hB(k))\in G/B\times\mathfrak{g}\times G/B$ is the subspace
\[ T_{gB(k)}G/B\oplus (g\mathfrak{b}g^{-1}\cap h\mathfrak{b}h^{-1})\oplus T_{hB(k)}G/B\]
of $T_{gB(k)}G/B\oplus\mathfrak{g}\oplus T_{hB(k)}G/B$.
\end{lemm}

Let $T$ be a maximal split torus in $G$ and let $\mathfrak{t}$ be its Lie algebra. The following result will prove essential later. Let us write $(G/B)^T$ for the set of $k$-points of $G/B$ which are fixed by the group $T(k)$. Note that this set is in bijection with the Weyl group $W$ of $(G,T)$.

\begin{theo}\label{intersectionBorels}
Let $hB(k)\in (G/B)(k)$. We have
\[ \sum_{gB(k)\in (G/B)^T}d\pi_2(T_{(gB(k),0,hB(k))}(\gtilde\times_{\mathfrak{g}}\gtilde))=T_{(0,hB(k))}\gtilde\]
\end{theo}

\begin{proof}
Let us remark that $gB(k)\in (G/B)^T$ if and only if $T\subset gB(k)g^{-1}$ which is equivalent to $\mathfrak{t}\subset g\mathfrak{b}g^{-1}$. Let $\mathcal{B}$ be the set of Borel sub-algebras of $\mathfrak{g}$ containing $\mathfrak{t}$. Using Lemmas \ref{tangent1} and \ref{tangent2}, we see that the statement is equivalent to the following identity :
\begin{equation}\label{mainequation} \sum_{\mathfrak{b'}\in\mathcal{B}}(\mathfrak{b'}\cap h\mathfrak{b}h^{-1})=h\mathfrak{b}h^{-1}.\end{equation}
This, in turn, is a consequence of Theorem \ref{enveloppeBorels}.
\end{proof}

\section{Local deformation rings}

Let $\mathcal{C}$ be the category of finite local $L$-algebras $A$ with residue field isomorphic to $L$. If $A$ is an object of $\mathcal{C}$ we denote by $\mathfrak{m}_A$ its unique maximal ideal.

\subsection{$(\varphi,\Gamma_K)$-modules}

Let $K'$ be the maximal unramified extension of $\Qp$ contained in $K(\mu_{p^{\infty}})$, it is a finite extension of $\Qp$. Let $\mathcal{R}$ be the Robba ring of $K$ defined as $\varinjlim_{r<1}\mathcal{R}^{]r,1[}$ where $\mathcal{R}^{]r,1[}$ is the ring of rigid analytic functions on the open annulus $\{r<|X| <1\}$ over $K'$. This ring is a Bezout domain. 
If $A$ is a finite dimensional $\Qp$-algebra let $\mathcal{R}_A\coloneqq A\otimes_{\Qp}\mathcal{R}$. The ring $\mathcal{R}$ is endowed with a Frobenius endomorphism $\phi$ and a continuous action of the group $\Gamma_K\coloneqq \Gal(K(\zeta_{p^\infty})/K)$ commuting with $\phi$ (see \cite[Def.~2.2.2]{KPX}). The ring $\mathcal{R}$ contains an element $t$ which is the image in $\mathcal{R}$ of the rigid analytic function $x\mapsto\log(1+x)$ defined over the open unit disc over $\Qp$. This element has the properties $\phi(t)=pt$ and $[\rec_K(a)]\cdot t=\chi_{\cyc}(\rec_K(a))t=at$, for $a\in K^\times$.

When $A$ is an object of $\mathcal{C}$, we define a $(\varphi,\Gamma_K)$-module over $\mathcal{R}_A$ as a pair $(\mathcal{D}_A,\varphi)$ where $\mathcal{D}_A$ is a finite free $\mathcal{R}_A$-module, $\varphi$ is a $\phi$-semilinear endomorphism of $\mathcal{D}_A$ inducing an isomorphism $\mathcal{R}_A\otimes_{\mathcal{R}_A,\phi}\mathcal{D}_A\xrightarrow{\sim} \mathcal{D}_A$, and $\mathcal{D}_A$ is equipped with a continuous semilinear action of $\Gamma_K$ commuting with $\varphi$ (here $\mathcal{D}_A$ is a $\mathcal{R}$-module of finite type and has the canonical topology coming from the topology of $\mathcal{R}$). As $A$ is a finite local $\Qp$-algebra, this definition coincides with \cite[Def.~2.2.12]{KPX}.

If $\mathcal{D}_1$ and $\mathcal{D}_2$ are two $(\varphi,\Gamma_K)$-modules over $\mathcal{R}_A$. There is a $(\varphi,\Gamma_K)$-module $\Hom(\mathcal{D}_1,\mathcal{D}_2)$ defined over $\mathcal{R}_A$ whose underlying $\mathcal{R}_A$-module is the space of $\mathcal{R}_A$-linear maps from $\mathcal{D}_1$ to $\mathcal{D}_2$. Namely, $\mathcal{R}$ is a flat $\mathcal{R}$-module via $\phi$, so that the canonical map $\mathcal{R}\otimes_{\mathcal{R},\phi}\Hom_{\mathcal{R}}(M_1,M_2)\rightarrow\Hom_{\mathcal{R}}(\mathcal{R}\otimes_{\mathcal{R},\phi}M_1,\mathcal{R}\otimes_{\mathcal{R},\phi}M_2)$ is an isomorphism, which isomorphism is used to define $\varphi$ on $\Hom_{\mathcal{R}}(M_1,M_2)$.

For $i\geq 0$, the $i$-th cohomology group $H^i_{(\varphi,\Gamma_K)}(\mathcal{D})$ of a $(\varphi,\Gamma_K)$-module $\mathcal{D}$ is defined in \cite[3.1]{LiuCohomology}. If $\mathcal{D}_A$ is a $(\varphi,\Gamma_K)$-module over $\mathcal{R}_A$, it follows from \cite[Thm.~5.3]{LiuCohomology} that $H^i_{(\varphi,\Gamma)}(\mathcal{D}_A)$ is of finite type over $A$ and zero for $i> 2$.

For any continuous group homomorphism $\delta : K^\times\rightarrow L^\times$, we recall that we can construct a rank one $(\varphi,\Gamma_K)$-module $\mathcal{R}_L(\delta)$ over $\mathcal{R}_L$ such that the map $\delta\mapsto \mathcal{R}_L(\delta)$ induces a bijection between the set of continuous group homomorphisms $K^\times\rightarrow L^\times$ and the set of isomorphism classes of rank one $(\varphi,\Gamma_K)$-modules over $\mathcal{R}_L$ (see \cite[\S6.1]{KPX} for the precise construction of $\mathcal{R}_L(\delta)$).

By definition a $(\varphi,\Gamma_K)$-module over $\mathcal{R}[\tfrac{1}{t}]$ is a finite free $\mathcal{R}[\tfrac{1}{t}]$-module $\mathcal{M}$ with a $\phi$-semilinear endomorphism $\varphi$ and a semilinear action of $\Gamma_K$ such that there exists a sub-$\mathcal{R}$-module $\mathcal{D}$ of $\mathcal{M}$ which is stable by $\varphi$ and $\Gamma_K$, generates $\mathcal{M}$ as a $\mathcal{R}[\tfrac{1}{t}]$-module and is a $(\varphi,\Gamma_K)$-module over $\mathcal{R}$.

\begin{lemm}\label{stableparsousobjet}
Let $\mathcal{M}$ be $(\varphi,\Gamma_K)$-module over $\mathcal{R}[\tfrac{1}{t}]$ and let $\mathcal{N}$ a sub-$\mathcal{R}[\tfrac{1}{t}]$-module of $\mathcal{M}$ which a direct factor as $\mathcal{R}[\tfrac{1}{t}]$-module and stable under $\varphi$ and $\Gamma_K$. Then $\mathcal{N}$ is a $(\varphi,\Gamma_K)$-module over $\mathcal{R}[\tfrac{1}{t}]$.
\end{lemm}

\begin{proof}
Let $\mathcal{D}$ be a sub-$\mathcal{R}$-module of $\mathcal{M}$, which is a $(\varphi,\Gamma_K)$-module and generates $\mathcal{M}$ as a $\mathcal{R}[\tfrac{1}{t}]$-module. It is sufficient to prove that $\mathcal{D}'\coloneqq \mathcal{D}\cap\mathcal{N}$ is a $(\varphi,\Gamma_K)$-module over $\mathcal{R}$. It follows from \cite[Lemme~4.13]{Bergereqdiff} that $\mathcal{D}'$ is a finite free $\mathcal{R}$-module. Consequently we have to prove that the $\mathcal{R}$-linear map $\mathcal{R}\otimes_{\mathcal{R},\phi}\mathcal{D}'\rightarrow\mathcal{D}'$ given by the restriction of $\varphi$ is an isomorphism. 
The snake lemma applied to the following morphism of short exact sequences
\[ \begin{tikzcd}0\ar[r] & \mathcal{R}\otimes_{\mathcal{R},\phi}\mathcal{D}'\ar[r]\ar[d, "\varphi_{\mathcal{D}'}"] & \mathcal{R}\otimes_{\mathcal{R},\phi}\mathcal{D}\ar[r]\ar[d,"\varphi_{\mathcal{D}}"]& \mathcal{R}\otimes_{\mathcal{R},\phi}(\mathcal{D}/\mathcal{D}')\ar[r]\ar[d,"\varphi_{\mathcal{D}/\mathcal{D}'}"] & 0 \\
0 \ar[r] & \mathcal{D}'\ar[r] & \mathcal{D}\ar[r]& \mathcal{D}/\mathcal{D}'\ar[r] & 0
\end{tikzcd} ,\]
where $\varphi_{\mathcal{D}}$ is bijective by the very definition of a $(\varphi,\Gamma_K)$-module, shows that the map $\varphi_{\mathcal{D}/\mathcal{D}'}$ is surjective and that it is enough to check it is also injective to get that $\varphi_{\mathcal{D}'}$ is an isomorphism.
Since $\mathcal{D}/\mathcal{D}'$ is $\mathcal{R}$-torsion free, its kernel is torsion free and a comparison of the ranks then shows that $\varphi_{\mathcal{D}/\mathcal{D}'}$ is injective as the ring $\mathcal{R}$ is Bezout.
\end{proof}

If $A$ is an object of $\mathcal{R}_A$, we define a $(\varphi,\Gamma_K)$-module over $\mathcal{R}_A[\tfrac{1}{t}]$ as being a $(\varphi,\Gamma_K)$-module $\mathcal{M}$ over $\mathcal{R}[\tfrac{1}{t}]$ together with a morphism of $\Qp$-algebras from $A$ into $\End_{\varphi,\Gamma_K}\mathcal{M}$ such that $\mathcal{M}$ is a finite free $\mathcal{R}_A[\tfrac{1}{t}]$-module.

\subsection{Filtered deformation functors}

Let $A$ be an object of $\mathcal{C}$ and let $\mathcal{D}_A$ be a $(\varphi,\Gamma_K)$-module over $\mathcal{R}_A$. We define a filtration $\mathcal{F}$ of $\mathcal{D}_A$ as a sequence
\[ 0=\mathcal{F}_0\subset\mathcal{F}_1\subset\cdots\subset\mathcal{F}_m=\mathcal{D}_A\]
of sub-$(\varphi,\Gamma_K)$-modules of $\mathcal{D}_A$ such that each $\mathcal{F}_i$ is a direct factor of $\mathcal{D}_A$ as an $\mathcal{R}_A$-module. When $m=\rk_{\mathcal{R}_A} \mathcal{D}_A$ and each quotient $\mathcal{F}_i/\mathcal{F}_{i-1}$ is of rank $1$ over $\mathcal{R}_A$, we say that $\mathcal{F}$ is a \emph{triangulation} of $\mathcal{D}_A$.

Let $\mathcal{F}$ be a triangulation of a $(\varphi,\Gamma_K)$-module $\mathcal{D}$ over $\mathcal{R}_L$. For $1\leq i\leq d$, let $\delta_i$ the unique continuous morphism $K^\times\rightarrow L^\times$ such that $\mathcal{F}_i/\mathcal{F}_{i-1}\simeq\mathcal{R}_L(\delta_i)$. The character $\delta_1\otimes\cdots\otimes\delta_d$ from $(K^\times)^d$ to $L^\times$ depends only on $\mathcal{D}$ and $\mathcal{F}$ and is called the \emph{parameter} of the triangulation $\mathcal{F}$.

If $\mathbf{k}=(k_{\tau})_{\tau\in\Sigma}\in\Z^{\Sigma}$, we note $z^{\mathbf{k}}$ the character $z\mapsto \prod_{\tau\in\Sigma}\tau(z)^{k_{\tau}}$ from $K^{\times}$ into $L^{\times}$. A character $\delta_1\otimes\cdots\otimes\delta_d$ of $(K^\times)^d$ is called \emph{regular} if, for all $i\neq j$, we have
\begin{equation}\label{regchar}
 \delta_i\delta_j^{-1}\notin\{ z^{\mathbf{k}}, z^{\mathbf{k}}\varepsilon_K ; \, \mathbf{k}\in\Z^{[K:\Qp]}\}.
 \end{equation}

From now we fix $\mathcal{D}$ a $(\varphi,\Gamma_K)$-module over $\mathcal{R}_L$ and $\mathcal{F}$ a filtration of $\mathcal{D}$. If $A$ is an object of $\mathcal{C}$, we define $\mathfrak{X}_{\mathcal{D},\mathcal{F}}(A)$ as the set of isomorphism classes of triples $(\mathcal{D}_A,\pi,\mathcal{F}_A)$ where $\mathcal{D}_A$ is a $(\varphi,\Gamma_K)$-module over $\mathcal{R}_A$, $\pi$ is an $\mathcal{R}_A$-linear map from $\mathcal{D}_A$ to $\mathcal{D}$ commuting to $\varphi$ and $\Gamma_K$, inducing an isomorphism $\mathcal{D}_A\otimes_AL\xrightarrow{\sim} \mathcal{D}$, and $\mathcal{F}_A=(\mathcal{F}_{A,i})_{0\leq i\leq m}$ is a filtration of $\mathcal{D}_A$ such that $\pi(\mathcal{F}_{A,i})=\mathcal{F}_i$ for all $0\leq i\leq m$. This construction can be promoted naturally into a functor from $\mathcal{C}$ to the category of sets. In the case $K=\Qp$, the functor $\mathfrak{X}_{\mathcal{D},\mathcal{F}}$ was defined by Chenevier in \cite{Chefougere}. Below we will make reference to the statements \cite[Prop.~3.4]{Chefougere} and \cite[Prop.~3.6 (i) and (iii)]{Chefougere} which concern only the case $K=\Qp$. However their statement and proof extend verbatim to the general case where $K$ is a finite extension of $\Qp$ so that we will apply them without more explanation to our situation. It follows from \cite[Prop.~3.4]{Chefougere} that the functor $\mathfrak{X}_{\mathcal{D},\mathcal{F}}$ admits a versal deformation $L$-algebra, i.e.~a complete noetherian local $L$-algebra $R$ such that
\[ \Hom_{\mathrm{pro}-\mathcal{C}}(R,-)\simeq\mathfrak{X}_{\mathcal{D},\mathcal{F}}.\]
When $\mathcal{F}=(0\subset \mathcal{D})$, we simply write $\mathfrak{X}_\mathcal{D}$ for the functor $\mathfrak{X}_{\mathcal{D},\mathcal{F}}$, which then coincides with the deformation functor of $\mathcal{D}$. There is a natural map of functors $\mathfrak{X}_{\mathcal{D},\mathcal{F}}\rightarrow\mathfrak{X}_\mathcal{D}$ which is defined by $(\mathcal{D}_A,\pi,\mathcal{F}_A)\mapsto (\mathcal{D}_A,\pi)$. If we assume in addition that $\Hom_{(\varphi,\Gamma_K)}(\gr_i(\mathcal{D}),\mathcal{D}/\mathcal{F}_i)=0$ for all $i$, then \cite[Prop.~3.6.(i)]{Chefougere} shows that the map of functors $\mathfrak{X}_{\mathcal{D},\mathcal{F}}\rightarrow\mathfrak{X}_\mathcal{D}$ is injective and therefore we can identify $\mathfrak{X}_{\mathcal{D},\mathcal{F}}$ with a subfunctor of $\mathfrak{X}_\mathcal{D}$. 
In the particular case that $\mathcal{F}$ is a triangulation, the functor $\mathfrak{X}_{\mathcal{D},\mathcal{F}}$ was introduced in \cite[Def.~2.3.2]{BelChe}; then the map of functors $\mathfrak{X}_{\mathcal{D},\mathcal{F}}\rightarrow\mathfrak{X}_{\mathcal{D}}$ is relatively representable if we assume  $\Hom_{(\varphi,\Gamma_K)}(\gr_i(\mathcal{D}),\mathcal{D}/\mathcal{F}_i)=0$ for all $i$ (see \cite[Prop.~2.3.9]{BelChe}), an assumption that is satisfied if $\mathcal{F}$ is a triangulation of $\mathcal{D}$ with regular parameter as it follows from \cite[Prop.~3.10.(1)]{LiuCohomology}.
If $\mathcal{F}$ is a filtration of $\mathcal{D}$, let $\End_{\mathcal{F}} \mathcal{D}$ be the sub-$\mathcal{R}_L$-module of $\End \mathcal{D}$ whose elements are $\mathcal{R}_L$-linear maps respecting $\mathcal{F}$. It is a sub-$(\varphi,\Gamma_K)$-module of $\End \mathcal{D}$. It follows from \cite[Prop.~3.6.(iii)]{Chefougere} that if $H^2_{(\varphi,\Gamma_K)}(\End_{\mathcal{F}}\mathcal{D})=0$, the functor $\mathfrak{X}_{\mathcal{D},\mathcal{F}}$ is formally smooth. In particular if $H^2_{(\varphi,\Gamma_K)}(\End \mathcal{D})=0$, the functor $\mathfrak{X}_\mathcal{D}$ is formally smooth, which implies that a versal deformation ring for $\mathfrak{X}_\mathcal{D}$ is a formally smooth complete noetherian local $L$-algebra with residual field isomorphic to $L$, i.e.~of the form $L\dbl X_1,\dots, X_d\dbr$ for some non-negative integer $d$.

For the purpose of this paper we need an other kind of deformation problem. Let $A$ be an object of $\mathcal{C}$ and $\mathcal{D}_A$ a $(\varphi,\Gamma_K)$-module over $\mathcal{R}_A$. The properties of the element $t\in\mathcal{R}$ show that the endomorphism $\phi$ and the action of $\Gamma_K$ extends canonically to the ring $\mathcal{R}_A[\tfrac{1}{t}]$ and, if $\mathcal{D}$ is a $(\varphi,\Gamma_K)$-module over $\mathcal{R}_A$, there are canonical semilinear extensions of $\varphi$ and of the action of $\Gamma_K$ to $\mathcal{D}[\tfrac{1}{t}]$. A filtration of $\mathcal{D}_A[\frac{1}{t}]$ is a sequence
\[ \mathcal{M}=\left(0\subset\mathcal{M}_1\subset\cdots\mathcal{M}_m=\mathcal{D}_A\left[\tfrac{1}{t}\right]\right)\]
by sub-$\mathcal{R}_A[\tfrac{1}{t}]$-modules which are direct factors and are stable under $\varphi$ and $\Gamma_K$. 
\begin{rema}
If $\mathcal{F}=(\mathcal{F}_i)_{0\leq i\leq m}$ is a filtration of $\mathcal{D}_A$, the family $\mathcal{F}[\tfrac{1}{t}]\coloneqq (\mathcal{F}_i[\tfrac{1}{t}])$ is a filtration of $\mathcal{D}_A[\tfrac{1}{t}]$. However, if $(\mathcal{M}_i)_{0\leq i\leq m}$ is a filtration of $\mathcal{D}_A[\tfrac{1}{t}]$, the family $(\mathcal{M}_i\cap \mathcal{D}_A)_{0\leq i\leq m}$ need not be a filtration of $\mathcal{D}_A$ since the $\mathcal{R}_A$-modules $\mathcal{M}_i\cap \mathcal{D}_A$ may fail to be projective. 
\end{rema}
A family of the form $(\mathcal{M}_i\cap \mathcal{D}_A)_{0\leq i\leq m}$ , for a given filtration $\mathcal{M}_i$ of $\mathcal{D}_A[\tfrac{1}{t}]$, is what we call loosely an unsaturated filtration of $\mathcal{D}_A$. When $A=L$, the fact that $\mathcal{R}$ is a Bezout domain implies that the family $\mathcal{M}\cap\mathcal{D}_A\coloneqq (\mathcal{M}_i\cap \mathcal{D}_A)_{0\leq i\leq m}$ is actually a filtration of $\mathcal{D}_A$ and the map $\mathcal{M}\mapsto\mathcal{M}\cap \mathcal{D}_A$ is a bijection from the set of filtrations of $\mathcal{D}_A[\tfrac{1}{t}]$ onto the set of filtrations of $\mathcal{D}_A$ whose inverse is $\mathcal{F}\mapsto\mathcal{F}[\tfrac{1}{t}]$.

Let $\mathcal{D}$ be a $(\varphi,\Gamma_K)$-module over $\mathcal{R}_L$ and let $\mathcal{M}$ be a filtration of $\mathcal{D}[\tfrac{1}{t}]$. If $A$ is an object of $\mathcal{C}$, we define $\mathfrak{X}_{\mathcal{D},\mathcal{M}}(A)$ as the set of isomorphism classes of triples $(\mathcal{D}_A,\pi,\mathcal{M}_A)$ where $\mathcal{D}_A$ is a $(\varphi,\Gamma_K)$-module over $\mathcal{R}_A$, $\pi_A$ is a $(\varphi,\Gamma_K)$-module morphism $\mathcal{D}_A\rightarrow \mathcal{D}$ inducing an isomorphism $L\otimes_A\mathcal{D}_A\xrightarrow{\sim} \mathcal{D}$ and $\mathcal{M}_A$ is a filtration of $\mathcal{D}_A[\tfrac{1}{t}]$ such that $\pi(\mathcal{M}_{A,i})=\mathcal{M}_i$. The construction $A\mapsto\mathfrak{X}_{\mathcal{D},\mathcal{M}}(A)$ can be promoted into a functor from $\mathcal{C}$ to the category of sets. When $\mathcal{F}\coloneqq \mathcal{M}\cap \mathcal{D}$ we can check that the map $(\mathcal{D}_A,\pi_A,\mathcal{F}_A)\mapsto (\mathcal{D}_A,\pi_A,\mathcal{F}_A[\tfrac{1}{t}])$ induces an injection of functors $\mathfrak{X}_{\mathcal{D},\mathcal{F}}\hookrightarrow\mathfrak{X}_{\mathcal{D},\mathcal{M}}$ and we use it to identify $\mathfrak{X}_{\mathcal{D},\mathcal{F}}$ with a subfunctor of $\mathfrak{X}_{\mathcal{D},\mathcal{M}}$.

\begin{rema}
When $\mathcal{M}$ is a triangulation of $\mathcal{D}$, the functor $\mathfrak{X}_{\mathcal{D},\mathcal{M}}$ coincides with the functor of isomorphism classes of the groupoid $X_{\mathcal{D},\mathcal{M}}$ introduced in \cite[\S~3.5]{BHS3}.
\end{rema}

The following statement is a direct consequence of the definitions; we state it for the sake of completeness and comfort of reading.

\begin{scol}
Let $\pi:\,\mathfrak{X}\rightarrow\mathfrak{Y}$ be a relatively representable morphism between functors from $\mathcal{C}$ to the category of sets. If $\mathfrak{Y}$ admits a versal deformation $L$-algebra, then $\mathfrak{X}$ admits a versal deformation $L$-algebra. More precisely if $\Spf R\rightarrow\mathfrak{Y}$ is a hull for $\mathfrak{Y}$ then the functor $\Spf R\times_{\mathfrak{Y}}\mathfrak{X}$ is pro-representable by a local complete noetherian $L$-algebra $S$ and $\Spf S\rightarrow\mathfrak{X}$ is a hull for $\mathfrak{X}$.
\end{scol}

We will also need the following fact which is a direct consequence of \cite[Prop.~3.4.6]{BHS3}.

\begin{prop}\label{genericfiltrations}
Let $\mathcal{D}$ be a $(\varphi,\Gamma_K)$-module and let $\mathcal{F}$ be a triangulation of $\mathcal{D}$ whose parameter is regular in the sense of $(\ref{regchar})$. Let $\mathcal{M}\coloneqq \mathcal{F}[\tfrac{1}{t}]$. Then the forgetful map $\mathfrak{X}_{\mathcal{D},\mathcal{M}}\rightarrow\mathfrak{X}_{\mathcal{D}}$ is injective and relatively representable. This implies that $\mathfrak{X}_{\mathcal{D},\mathcal{M}}$ admits a versal deformation $L$-algebra.
\end{prop}

This means that the map $\mathfrak{X}_{\mathcal{D},\mathcal{M}}\rightarrow\mathfrak{X}_{\mathcal{D}}$ is injective and that for all objects $A$ in $\mathcal{C}$ and all $x\in\mathfrak{X}_{\mathcal{D}}(A)$, there is a unique quotient $A_x$ of $A$ such that for any map $A\rightarrow B$ in $\mathcal{C}$, the image of $x$ in $\mathfrak{X}_{\mathcal{D}}(B)$ is in $\mathfrak{X}_{\mathcal{D},\mathcal{M}}(B)$ if and only if the map $A\rightarrow B$ factors through $A_x$.

\subsection{Crystalline $(\varphi,\Gamma_K)$-modules}\label{crystalline}

Let $K_0=W(k_K)[\tfrac{1}{p}]$, let $\sigma$ the absolute Frobenius automorphism of $K_0$ and $f=[k_K:\Fp]=[K_0:\Qp]$. If $A$ is an object of $\mathcal{C}$, an \emph{isocrystal} over $k_K$ with coefficients in $A$ is a pair $(V,\varphi)$ where $V$ is a finite projective $A\otimes_{\Qp}K_0$-module and $\varphi$ is an $\Id_A\otimes\sigma$-semilinear automorphism of $V$. Actually these conditions automatically imply that $V$ is a finite free $A\otimes_{\Qp}K_0$-module. Its rank is by definition its rank as a $A\otimes_{\Qp}K_0$-module. If $(V,\varphi)$ is an isocrystal over $k_K$ with coefficients in $A$, we define $\chi(V,\varphi)$ as the characteristic polynomial of the $A\otimes_{\Qp}K_0$-linear endomorphism $\varphi^f$. This polynomial is invariant under $\Id_A\otimes\sigma$, to the effect that $\chi(V,\varphi)$ lies in $A[X]$. Assume now that $A=L$. If $\chi(V,\varphi)=PQ$ where $P$ and $Q$ are coprime elements in $L[X]$, then there exists a unique $\varphi$-stable $L\otimes_{\Qp}K_0$-submodule $W\subset V$ such that $\chi(W,\varphi|_W)=P$. Actually we have explicitely $W=\ker P(\varphi^f)$.

Recall that there exists a left exact functor $D_{\cris}$ from the category of $(\varphi,\Gamma_K)$-modules over $\mathcal{R}_L$ to the category of isocrystals over $k_K$ with coefficients in $L$. It is defined by $D_{\cris}(\mathcal{D})\coloneqq \mathcal{D}[\tfrac{1}{t}]^{\Gamma_K}$ (see \cite[2.2.7.]{BelChe} for the case where $K=\Qp$). We say that a $(\varphi,\Gamma_K)$-module $\mathcal{D}$ over $\mathcal{R}_L$ is \emph{crystalline} if $\rk_{K_0} D_{\cris}(\mathcal{D})=\rk_{\mathcal{R}}\mathcal{D}$. Let $\mathcal{D}$ be a crystalline $(\varphi,\Gamma_K)$-module over $\mathcal{R}_L$. Arguing as in \cite[2.4.2.]{BelChe}, there exists a bijection between sub-$(\varphi,\Gamma_K)$-modules of $\mathcal{D}$ which are direct summands as $\mathcal{R}_L$-module and $\varphi$-stable sub-$L\otimes_{\Qp}K_0$-modules of $D_{\cris}(\mathcal{D})$.

A \emph{refinement} of a rank $d$ isocrystal $(D,\varphi)$ over $k$ is a filtration $F=(F_i)_{0\leq i\leq d}$ of $D$
\[ F_0=0\subsetneq F_1 \subsetneq \cdots \subsetneq F_d=D\]
such that each $F_i$ is a $L\otimes_{\Qp}K_0$-submodule stable under $\varphi$. Note that each $F_i$ is necessarily free over $L\otimes_{\Qp}K_0$ and consequently of rank $i$. 

If $\mathcal{D}$ is a crystalline $(\varphi,\Gamma_K)$-module over $\mathcal{R}_L$, there is consequently a bijection $\mathcal{F}=(\mathcal{F}_i)_{0\leq i\leq d}\mapsto D_{\cris}(\mathcal{F})\coloneqq (D_{\cris}(\mathcal{F}_i))_{0\leq i\leq d}$ between the set of triangulations of $\mathcal{D}$ and the set of refinements of $D_{\cris}(\mathcal{D})$. A refinement $F$ of a $k$-isocrystal $(D,\varphi)$ gives rise to a decomposition of $\chi(V,\varphi)$ as a product of polynomials of degree one
\[ \chi(D,\varphi)=\prod_{i=1}^n\chi(F_i/F_{i-1},\varphi). \]
In particular $\chi(D,\varphi)$ is split over $L$ and the triangulation defines an ordering $(\phi_1,\dots,\phi_d)$ of the roots of $\chi(D,\varphi)$ such that $\chi(F_i,\varphi)=\prod_{j=1}^i(X-\phi_j)$. We define $\delta_{F}$ to be the unramified character $T(K)\rightarrow L^{\times}$ given by the formula
\[ (a_1,\dots, a_d)\longmapsto \prod_{i=1}^d \phi_i^{v_K(a_i)} .\]
If $(D,\varphi)=D_{\cris}(\mathcal{D})$ for a crystalline $(\varphi,\Gamma_K)$-module $\mathcal{D}$ over $\mathcal{R}_L$ and if $\mathcal{F}$ is a triangulation of $\mathcal{D}$, we define $\delta_{\mathcal{F}}\coloneqq \delta_{D_{\cris}(\mathcal{F})}$. It follows from the classification of sub-$(\varphi,\Gamma_K)$-modules of rank one $(\varphi,\Gamma_K)$-modules (\cite[Prop.~6.2.8.(1)]{KPX}) that the parameter of the triangulation $\mathcal{F}$ is the product of $\delta_{\mathcal{F}}$ with an algebraic character of $(K^\times)^d$.

Conversely if the polynomial $\chi(D,\varphi)$ is separable and split in $L[X]$, each ordering $(\phi_1,\dots,\phi_d)$ of its roots comes from a unique refinement of $D$. In this case, the character $\delta_{F}$ completely determines the refinement $F$.

We say that a crystalline $(\varphi,\Gamma_K)$-module $\mathcal{D}$ over $\mathcal{R}_L$ is \emph{$\varphi$-generic} if the polynomial $\chi(D_{\cris}(\mathcal{D}))$ is separable split over $L$ with pairwise distinct roots $(\phi_1,\dots,\phi_d)$ such that $\phi_i\phi_j^{-1}\neq p^f$ for $i\neq j$. This property in particular implies that for each triangulation $\mathcal{F}$ of $\mathcal{D}$, the parameter of $\mathcal{F}$ is regular so that the assumption on the triangulation $\mathcal F$ in Proposition \ref{genericfiltrations} is satisfied. As a consequence we have the following relatively representable inclusions
\[ \mathfrak{X}_{\mathcal{D},\mathcal{F}}\subset\mathfrak{X}_{\mathcal{D},\mathcal{F}[\frac{1}{t}]}\subset\mathfrak{X}_{\mathcal{D}},\]
where both $\mathfrak{X}_{\mathcal{D}}$ and $\mathfrak{X}_{\mathcal{D},\mathcal{F}}$ are formally smooth, unlike the functor $\mathfrak{X}_{\mathcal{D},\mathcal{F}[\frac{1}{t}]}$ which does not share the property in full generality.

Let $\mathcal{D}$ be a crystalline $(\varphi,\Gamma_K)$-module over $\mathcal{R}_L$. For an object $A$ of $\mathcal{C}$, let $\mathfrak{X}_{\mathcal{D}}^{\cris}(A)$ the subset of $\mathfrak{X}_{\mathcal{D}}(A)$ of isomorphism classes of pairs $(\mathcal{D}_A,\pi_A)$ with $\mathcal{D}_A$ a crystalline $(\varphi,\Gamma_K)$-module. The subfunctor $A\mapsto \mathfrak{X}_{\mathcal{D}}^{\cris}(A)$ of $\mathfrak{X}_{\mathcal{D}}$ is simply denoted $\mathfrak{X}_{\mathcal{D}}^{\cris}$. If $\mathcal{D}_A$ is crystalline, the $A\otimes_{\Qp}K_0$-module $D_{\cris}(\mathcal{D}_A)$ is finite free of rank $\rk_{\mathcal{R}_L}\mathcal{D}$. Assume moreover that $\mathcal{D}$ is $\varphi$-generic crystalline. Let $\mathcal{F}$ be a triangulation of $\mathcal{D}$ with associated refinement $F=D_{\cris}(\mathcal{F})$ and set $\mathcal{M}=\mathcal{F}[\tfrac{1}{t}]$.

\begin{lemm}\label{filtrationcristalline}
Let $A$ be an object of $\mathcal{C}$ and let $(\mathcal{D}_A,\pi_A)\in\mathfrak{X}_{\mathcal{D}}^{\cris}(A)$. There exists a unique complete flag $F_A$ of $A\otimes_{\Qp}K_0$-submodules of $D_{\cris}(\mathcal{D}_A)$ which is stable under $\varphi$ and reduces to $F$ modulo $\mathfrak{m}_A$.
\end{lemm}

\begin{proof}
By assumption, the polynomial $\chi(D_{\cris}(\mathcal{D}))$ is separable split in $L[X]$ so that we can write
\[ \chi(D_{\cris}(\mathcal{D}))=\prod_{i=1}^n(X-x_i)\]
and assume that the filtration $F$ is given by $F_i=\ker\prod_{j=1}^i(\varphi^f-x_j)$.

 Let $\chi_A(D_{\cris}(\mathcal{D}_A))\in A[X]$ be the characteristic polynomial of the $A\otimes_{\Qp}K_0$-linear endomorphism $\varphi^f$ of $D_{\cris}(\mathcal{D}_A)$. The reduction modulo $\mathfrak{m}_A$ of $\chi_A(D_{\cris}(\mathcal{D}_A))$ is the polynomial $\chi(D_{\cris}(\mathcal{D}))\in L[X]$ which is separable split in $L[X]$. Thus there exists a unique $(\tilde{x}_1,\dots,\tilde{x}_n)\in A^n$ such that $\chi_A(D_{\cris}(\mathcal{D}_A))=\prod_{i=1}^n(X-\tilde{x}_i)$ and, for $1\leq i\leq n$, $\tilde{x_i}\equiv x_i\mod \mathfrak{m}_A$. Considering the characteristic polynomials of the $\varphi^f|_{F_{A,i}}$ we can check that $F_{A,i}=\ker\prod_{j=1}^i(\varphi^f-\tilde{x}_j)$ defines the desired filtration. On the other hand any complete flag with the desired properties must fulfill this condition.
 \end{proof}

Let
\[ \mathcal{M}_A\coloneqq \left(\mathcal{R}_A\left[\tfrac{1}{t}\right]\otimes_{K_0\otimes_{\Qp}A} F_{A,1}\subsetneq\cdots\subsetneq \mathcal{R}_A\left[\tfrac{1}{t}\right]\otimes_{K_0\otimes_{\Qp}A} F_{A,n}=\mathcal{D}_A\left[\tfrac{1}{t}\right] \right) \]
where we used the canonical isomorphism (\cite[Thm.~0.2]{Bergereqdiff})
\[ \mathcal{R}_A\left[\tfrac{1}{t}\right]\otimes_{K_0\otimes_{\Qp}A} D_{\cris}(\mathcal{D}_A)\simeq\mathcal{D}_A\left[\tfrac{1}{t}\right] \]
and $F_A$ is the filtration whose existence is proved in Lemma \ref{filtrationcristalline}. Then $(\mathcal{D}_A,\pi_A,\mathcal{M}_A)$ is an element of $\mathfrak{X}_{\mathcal{D},\mathcal{M}}(A)$. This implies that we have a sequence of inclusions
\[ \mathfrak{X}_{\mathcal{D}}^{\cris}\subset\mathfrak{X}_{\mathcal{D},\mathcal{F}[\frac{1}{t}]}\subset\mathfrak{X}_{\mathcal{D}}.\]
\begin{rema}
We point out that in general $\mathfrak{X}_{\mathcal{D}}^{\cris}$ does not embed into $\mathfrak{X}_{\mathcal{D},\mathcal{F}}$. This is only true if we impose some conditions on the relative position of $\mathcal{F}$ with respect to the Hodge filtration (see below). 
\end{rema}

\subsection{$\BdR^+$-representations}\label{Bdr-repns}

 If $\mathcal{D}$ is a $(\varphi,\Gamma_K)$-module over $\mathcal{R}$, we note $W_{\dR}^+(\mathcal{D})$ the $\BdR^+$-representation of $\mathcal{G}_K$ constructed in \cite[Prop.~2.2.6.]{BergerBpaires}. The rank of the $\BdR^+$-module $W_{\dR}^+(\mathcal{D})$ is equal to the rank of the $\mathcal{R}$-module $\mathcal{D}$. We obtain an exact functor $W_{\dR}^+$ from the category of $(\varphi,\Gamma_K)$-modules over $\mathcal{R}$ to the category of $\BdR^+$-representations of $\mathcal{G}_K$.

If $A$ is an object of $\mathcal{C}$, an $A\otimes_{\Qp}\BdR^+$-representation of $\mathcal{G}_K$ is a finite free $A\otimes_{\Qp}\BdR^+$-module with a continuous semilinear action of $\mathcal{G}_K$. If $\mathcal{D}_A$ is a $(\varphi,\Gamma_K)$-module over $\mathcal{R}_A$, the $L\otimes_{\Qp}\BdR^+$-representation $W_{\dR}^+(\mathcal{D}_A)$ is actually an $A\otimes_{\Qp}\BdR^+$-module with a continuous semilinear action of $\mathcal{G}_K$. It follows from \cite[Lemma~3.3.5.(i)]{BHS3} that $W_{\dR}^+(\mathcal{D}_A)$ is a finite free $A\otimes_{\Qp}\BdR^+$-module and consequently an $A\otimes_{\Qp}\BdR^+$-representation of $\mathcal{G}_K$. If $W$ is an $L\otimes_{\Qp}\BdR^+$-representation of $\mathcal{G}_K$ we define $\mathfrak{X}_W$ the functor from $\mathcal{C}$ to the category of sets such that $\mathfrak{X}_W(A)$ is the set of equivalence classes of pairs $(W_A,\pi_A)$ where $W_A$ is a $A\otimes_{\Qp}\BdR^+$-representation of $\mathcal{G}_K$ and $\pi_A$ is an $A\otimes_{\Qp}\BdR^+$-linear and $\mathcal{G}_K$-equivariant morphism from $W_A$ to $W$ inducing an isomorphism $L\otimes_AW_A\xrightarrow{\sim} W$. If $\mathcal{D}$ is a $(\varphi,\Gamma_K)$-module over $\mathcal{R}_L$, the functor $W_{\dR}^+$ induces a map from $\mathfrak{X}_{\mathcal{D}}$ to $\mathfrak{X}_{W_{\dR}^+(\mathcal{D})}$.

Let $\mathcal{D}$ be a $(\varphi,\Gamma_K)$-module over $\mathcal{R}$. Let $D_{\dR}(\mathcal{D})\coloneqq (W_{\dR}^ +(\mathcal{D})\otimes_{\BdR^+}\BdR)^{\mathcal{G}_K}$. The \emph{de Rham filtration} on $D_{\dR}(\mathcal{F})$ is defined by
\[ \Fil^i_{\dR}(D_{\dR}(\mathcal{D}))\coloneqq (t^i W_{\dR}^+(\mathcal{D}))^{\mathcal{G}_K}\subset D_{\dR}(\mathcal{D}) \]
We say that $\mathcal{D}$ is \emph{de Rham} if $\dim_K D_{\dR}(\mathcal{D})=\rk_{\mathcal{R}}\mathcal{D}$. If $A$ is an object of $\mathcal{C}$ and if $\mathcal{D}_A$ is a $(\varphi,\Gamma_K)$-module over $\mathcal{R}_A$, then $D_{\dR}(\mathcal{D}_A)$ is a $A\otimes_{\Qp}K$-module. If we assume that $\mathcal{D}_A$ is de Rham, then it is finite free over $A\otimes_{\Qp}K$, and each $\Fil^i_{\dR}D_{\dR}(\mathcal{D})$ is a sub-$A\otimes_{\Qp}K$-module.

A \emph{filtered} $L\otimes_{\Qp}K$-module is a finite free $L\otimes_{\Qp}K$-module with a separated and exhaustive filtration by sub-$L\otimes_{\Qp}K$-modules. The functor $D_{\dR}$ is a left exact functor from the category of $(\varphi,\Gamma_K)$-modules over $L$ to the category of filtered $L\otimes_{\Qp}K$-modules. The restriction of the functor $D_{\dR}$ to the subcategory of de Rham $(\varphi,\Gamma_K)$-modules is exact and a crystalline $(\varphi,\Gamma_K)$-module over $\mathcal{R}_L$ is de Rham. Moreover there is a canonical isomorphism of $L\otimes_{\Qp}K$-modules $D_{\cris}(\mathcal{D})\otimes_{K_0}K\simeq D_{\dR}(\mathcal{D})$ (see for example \cite[Prop.~2.3.3]{BergerBpaires}).

Let $\mathcal{D}$ be a crystalline $(\varphi,\Gamma_K)$-module over $\mathcal{R}_L$. For all $\tau\in\Sigma$, we define
\[ D_{\dR,\tau}(\mathcal{D})\coloneqq L\otimes_{K\otimes_{\Qp}L,\tau}D_{\dR}(\mathcal{D})\]
It is a direct factor of $D_{\dR}(\mathcal{D})$ and we define a separated and exhaustive filtration on $D_{\dR,\tau}(\mathcal{D})$ by
\[ \Fil_{\dR,\tau}^iD_{\dR,\tau}(\mathcal{D})\coloneqq L\otimes_{K\otimes_{\Qp}L,\tau}(\Fil_{\dR}^iD_{\dR}(\mathcal{D}))\]
A \emph{Hodge-Tate type} is an element $\mathbf{k}=(\mathbf{k}_{\tau})_{\tau\in\Sigma}\in(\Z^n)^{[K:\Qp]}$ where each $\mathbf{k}_{\tau}$ is an increasing sequence of integers. We say that the Hodge-Tate type is \emph{regular} if all these sequences of integers are \emph{strictly} increasing. If $\mathcal{D}$ is a de Rham $(\varphi,\Gamma_K)$-module, \emph{its Hodge-Tate type} is by definition $(k_{1,\tau}\leq\cdots \leq k_{n,\tau})_{\tau\in\Sigma}$ where the $k_{i,\tau}$ are the integers $m$ such that $\gr^{-m}D_{\dR,\tau}(\mathcal{D})\neq0$, counted with multiplicity, where the multiplicity of $m$ is defined as the dimension  $\dim_L\gr^{-m}D_{\dR,\tau}(\mathcal{D})$.

Let $\mathcal{D}$ be a crystalline $(\varphi,\Gamma_K)$-module over $\mathcal{R}_L$ and let $\mathcal{F}$ be a triangulation of $\mathcal{D}$. We say that $\mathcal{F}$ is \emph{non critical}, if for all $1\leq i\leq \rk_{\mathcal{R}_L}\mathcal{D}$ and for all  $\tau\in\Sigma$, there exists some $m\in\Z$ such that
\[ (L\otimes_{K_0\otimes_{\Qp}L,\tau|_{K_0}}D_{\cris}(\mathcal{F}_i))\oplus\Fil_{\dR,\tau}^m=D_{\dR,\tau}(\mathcal{D}).\]
In this case we obviously have $m+i=\rk\, \mathcal{D}$.
With the help of the functor $W_{\dR}^+$, we can easily construct an exact functor $W_{\dR}$ from the category of $(\varphi,\Gamma_K)$-modules over $\mathcal{R}[\tfrac{1}{t}]$ to the category of $\BdR$-representations of $\mathcal{G}_K$ such that, for $\mathcal{D}$ a $(\varphi,\Gamma_K)$-module over $\mathcal{R}$, we have
\[ W_{\dR}\left(\mathcal{D}\left[\tfrac{1}{t}\right]\right)=W_{\dR}^+(\mathcal{D})\otimes_{\BdR^+}\BdR. \]

If $A$ is an object of $\mathcal{C}$, the image by $W_{\dR}$ of a $(\varphi,\Gamma_K)$-module over $\mathcal{R}_A[\tfrac{1}{t}]$ is a finite free as $A\otimes_{\Qp}\BdR$-module and consequently an $A\otimes_{\Qp}\BdR$-representation of $\mathcal{G}_K$ (see \cite[Lemma~3.3.5.(ii)]{BHS3}).

If $A$ is an object of $\mathcal{C}$ and $W_A$ is an $A\otimes_{\Qp}\BdR$-representation of $\mathcal{G}_K$ of rank $n$, we define a \emph{complete flag} of $W$ to be a filtration $(F_i)_{0\leq i\leq n}$ of $W$ by sub-$A\otimes_{\Qp}\BdR$-modules stable under $\mathcal{G}_K$ such that $F_i$ is a free $A\otimes_{\Qp}\BdR$-module of rank $i$.

 Let $W$ be a $L\otimes_{\Qp}\BdR^+$-representation of $\mathcal{G}_K$ and let $F$ be a complete flag of $W\otimes_{\BdR^+}\BdR$ stable under the action of $\mathcal{G}_K$. Let $A$ be an object of the category $\mathcal{C}$. A \emph{deformation of the pair} $(W,F)$ over $A$ is an element $(W_A,\pi_A,F_A)$ where $W_A$ is a $A\otimes_{\Qp}\BdR^+$-representation of $\mathcal{G}_K$, $\pi_A$ a $\mathcal{G}_K$-equivariant isomorphism from $W_A\otimes_AL$ to $W$ and $F_A$ a complete flag of $W_A\otimes_{\BdR^+}\BdR$ such that $F=(\pi_A\otimes\Id_{\BdR})(F_A)$. 
 We denote by $\mathfrak{X}_{W,F}$ the functor from the category $\mathcal{C}$ to the category of sets, that maps an object $A$ of $\mathcal{C}$ to the isomorphism class of deformations of $(W,F)$.

Let $\mathcal{D}_A$ be a $(\varphi,\Gamma_K)$-module over $\mathcal{R}_A$ and $\mathcal{M}_A$ a triangulation of $\mathcal{D}_A[\tfrac{1}{t}]$. It follows from Lemma \ref{stableparsousobjet} that each $\mathcal{M}_{A,i}$ is a $(\varphi,\Gamma_K)$-module over $\mathcal{R}_A[\tfrac{1}{t}]$. Thus
\[ W_{\dR}(\mathcal{M}_A)\coloneqq \left(W_{\dR}(\mathcal{M}_{A,0})\subset\cdots\subset W_{\dR}(\mathcal{M}_{A,n})=W_{\dR}\left(\mathcal{D}\left[\tfrac{1}{t}\right]\right)\right)\]
is a complete flag of $W_{\dR}^+(\mathcal{D}_A)\otimes_{\BdR^+}\BdR$. For $\mathcal{D}$ a $(\varphi,\Gamma_K)$-module over $\mathcal{R}_L$ and $\mathcal{F}$ a triangulation of $\mathcal{D}$, we deduce from this fact that the functor $W_{\dR}^+$ extends to a map of functors
\begin{equation}\label{foncteurWdR} \mathfrak{X}_{\mathcal{D},\mathcal{F}[\frac{1}{t}]}\longrightarrow\mathfrak{X}_{W_{\dR}^+(\mathcal{D}),W_{\dR}(\mathcal{F}[\frac{1}{t}])}\end{equation}
The following proposition is essentially \cite[Cor.~3.5.6.]{BHS3}.

\begin{prop}\label{smoothBHS}
If $\mathcal{D}$ is a $\varphi$-generic crystalline $(\varphi,\Gamma_K)$-module over $\mathcal{R}_L$ with regular Hodge-Tate type and $\mathcal{F}$ is a triangulation of $\mathcal{D}$, then the map \eqref{foncteurWdR} is formally smooth.
\end{prop}

\begin{prop}\label{kercris}
Let $\mathcal{D}$ be a $\varphi$-generic crystalline $(\varphi,\Gamma_K)$-module over $\mathcal{R}_L$ of regular Hodge-Tate type. Then, for all objects $A$ of $\mathcal{C}$, the preimage of the trivial $A\otimes_{\Q_p}\BdR^+$-representation of $\mathcal{G}_K$ under the map $W_{\dR}^+:\mathfrak{X}_{\mathcal{D}}(A)\rightarrow \mathfrak{X}_{W_{\dR}^+(\mathcal{D})}(A)$ identifies with $\mathfrak{X}_{\mathcal{D}}^{\cris}(A)$. Consequently the following sequence of $L$-vector spaces is exact
\[ 0\rightarrow T\mathfrak{X}_{\mathcal{D}}^{\cris}\rightarrow T\mathfrak{X}_{\mathcal{D}}\rightarrow T\mathfrak{X}_{W_{\dR}^+(\mathcal{D})}. \]
\end{prop}

\begin{proof}
An object $(\mathcal{D}_A,\pi_A)$ has a trivial image by $W_{\dR}^+$ if and only if $\mathcal{D}_A$ is a de Rham $(\varphi,\Gamma_K)$-module. We can conclude as in the proof of \cite[Cor.~2.7.(i)]{HellmSchrdensity}. Namely it follows from the $p$-adic monodromy theorem (\cite[Thm.~2.3.5.(1)]{BergerBpaires}) that $\mathcal{D}_A$ is a potentially semistable $(\varphi,\Gamma_K)$-module. Being an extension of finitely many cristalline representations, it is actually semistable. It follows from the $\varphi$-genericity assumption on $\mathcal{D}$ that no quotient of eigenvalues of $\varphi^f$ on $D_{\st}(\mathcal{D}_A)$ can be equal to $p^f$, so that the monodromy operator of $D_{\st}(\mathcal{D}_A)$ is trivial. Hence $\mathcal{D}_A$ is crystalline.
\end{proof}

\begin{coro}\label{exactcristriangulin}
Let $\mathcal{D}$ be a $\varphi$-generic crystalline $(\varphi,\Gamma_K)$-module over $\mathcal{R}_L$ of regular Hodge-Tate type and $\mathcal{F}$ a triangulation of $\mathcal{D}$. 
Then, for all objects $A$ of $\mathcal{C}$, the preimage of the trivial $A\otimes_{\Q_p}\BdR^+$-representation of $\mathcal{G}_K$ under the map \eqref{foncteurWdR} identifies with $\mathfrak{X}_{\mathcal{D}}^{\cris}(A)$.
Moreover the following sequence is exact
\[ 0\longrightarrow T\mathfrak{X}_{\mathcal{D}}^{\cris}\longrightarrow T\mathfrak{X}_{\mathcal{D},\mathcal{F}[\frac{1}{t}]}\longrightarrow T\mathfrak{X}_{W_{\dR}^+(\mathcal{D}),W_{\dR}(\mathcal{F}[\frac{1}{t}])}\longrightarrow0\]
\end{coro}

\begin{proof}
The first assertion is a direct consequence of Proposition \ref{kercris}. The second assertion follows after evaluation at $L[\varepsilon]$. The exactness on the right is a consequence of Proposition \ref{smoothBHS}.
\end{proof}

\subsection{Main theorem: the local version}
Let $\mathcal{D}$ be a $\varphi$-generic crystalline $(\varphi,\Gamma_K)$-module over $\mathcal{R}_L$. We write $\Tri(\mathcal{D})$ for the set of triangulations of $\mathcal{D}$, which is in bijection with the set of refinements of $D_{\cris}(\mathcal{D})$. The local theorem states as follows
\begin{theo}\label{mainlocalthm}
Let $\mathcal{D}$ be a $\varphi$-generic crystalline $(\varphi,\Gamma_K)$-module over $\mathcal{R}_L$ of regular Hodge-Tate type. Let $\Tri(\mathcal{D})$ be the set of triangulations of $\mathcal{D}$. Then the $L$-linear map
\[ \bigoplus_{\mathcal{F}\in\Tri(\mathcal{D})} T\mathfrak{X}_{\mathcal{D},\mathcal{F}[\frac{1}{t}]}\longrightarrow T\mathfrak{X}_{\mathcal{D}} \]
is surjective.
\end{theo}

\begin{rema}
The special case of the result where all refinements of $D$ are assumed non critical is a theorem due to G. Chenevier for $K=\Qp$ (\cite[Thm.~3.19]{Chefougere}) and to K. Nakamura for an arbitrary extension $K$ of $\Qp$ (\cite[Thm.~2.62.]{NakamuraGL2}).
\end{rema}

Before giving the proof of Theorem \ref{mainlocalthm}, let us recall some constructions and results from \cite{BHS3}, to which we refer the reader for relevent definitions if needed.

Let $W$ be an almost de Rham $L\otimes_{\Qp}\BdR^+$-representation of $\mathcal{G}_K$ (see \cite[3.1]{BHS3}). Let $i$ be a $L\otimes_{\Qp}K$-linear isomorphism $(L\otimes_{\Qp}K)^n\xrightarrow{\sim} D_{\pdR}(W)$. Let $\mathfrak{X}_W^\Box$ be the functor from $\mathcal{C}$ to the category of sets such that $\mathfrak{X}_W^\Box(W)$ is the set of isomorphism classes of triples $(W_A,\pi_A,i_A)$ where $W_A$ is some $A\otimes_{\Qp}\BdR^+$-representation of $\mathcal{G}_K$, $\pi_A$ is map from $W_A$ to $W$ inducing an isomorphism from $L\otimes_A W_A$ to $W$ and $i_A$ is an isomorphism between $(A\otimes_{\Qp}K)^n$ and $D_{\pdR}(W_A)$ compatible with $\pi_A$ and $i$. Let $\mathfrak{g}$ be the Lie algebra of the algebraic group $\GL_{n,K}$ and let $\tilde{\mathfrak{g}}\rightarrow \mathfrak{g}$ be Grothendieck's simultaneous resolution of singularities. As in section $\ref{tangentlocalmodel}$ we consider the scheme $X\coloneqq \gtilde\times_{\mathfrak{g}}\gtilde$. It follows from \cite[Lem.~3.2.2]{BHS3} that the forgetful map $\mathfrak{X}_W^\Box\rightarrow\mathfrak{X}_W$ is formally smooth, and  \cite[Thm.~3.2.5]{BHS3} states that the functor $\mathfrak{X}_W^\Box$ is pro-representable by the completion of $\gtilde_{K/\Qp,L}$ at the point $x=(0,i^{-1}(\Fil_{\dR}))$.

Let $F$ be a complete flag of $W\otimes_{\BdR^+}\BdR$ stable under $\mathcal{G}_K$. We can define $\mathfrak{X}_{W,F}$ as in section \ref{Bdr-repns} and $\mathfrak{X}_{W,F}^\Box$ by adding a framing of $D_{\pdR}(W_A)$ for $(W_A,\pi_A,F_A)\in \mathfrak{X}_{W,F}(A)$. The forgetful map $\mathfrak{X}_{W,F}^\Box\rightarrow \mathfrak{X}_{W,F}$ is then formally smooth and the functor $\mathfrak{X}_{W,F}^\Box$ is pro-representable by the completion of $X_{K/\Qp,L}$ at the point $x_{F}=(F_1,0,F_2)$ where $F_1=i^{-1}(D_{\pdR}(F))$ and $F_2=i^{-1}(\Fil_{\dR})$ (\cite[Cor.~3.5.8.(i)]{BHS3}). Moreover, the following diagram is commutative
\begin{equation}\label{diagramm1}
\begin{tikzcd} \mathfrak{X}_{W,F}^\Box \ar[r, "\mathrm{forget}"] \isoarrow{d} & \mathfrak{X}_W^\Box \isoarrow{d} \\
\widehat{X_{K/\Qp,L}}_{x_{F}}\ar[r,"\pi_2"] & \widehat{\gtilde_{K/\Qp,L}}_{x}
\end{tikzcd}
\end{equation}
where the upper horizontal map is the forgetful map and the lower horizontal map is induced by the second projection of $X$ on $\gtilde$. When $W=W_{\dR}^+(\mathcal{D})$ for a $(\varphi,\Gamma_K)$-module $\mathcal{D}$ and $F=W_{\dR}(\mathcal{M})$ for $\mathcal{M}$ a triangulation of $\mathcal{D}[\tfrac{1}{t}]$ we will use the shorter notation $x_{\mathcal{M}}$ in place of $x_{W_{\dR}(\mathcal{M})}$.

\begin{proof}[Proof of Theorem \ref{mainlocalthm}]
Let $W\coloneqq W_{\dR}^+(\mathcal{D})$. In a first step we prove that the $L$-linear map
\[ \bigoplus_{\mathcal{F}\in\Tri(\mathcal{D})}T\mathfrak{X}_{W,W_{\dR}(\mathcal{F}[\frac{1}{t}])} \longrightarrow T\mathfrak{X}_W \]
is surjective. Let's consider the commutative diagram
\[ \begin{tikzcd} \bigoplus_{\mathcal{F}\in\Tri(\mathcal{D})}T\mathfrak{X}^\Box_{W,W_{\dR}(\mathcal{F}[\frac{1}{t}])} \ar[r] \ar[d] & T\mathfrak{X}^\Box_{W} \ar[d] \\
\bigoplus_{\mathcal{F}\in\Tri(\mathcal{D})}T\mathfrak{X}_{W,W_{\dR}(\mathcal{F}[\frac{1}{t}])} \ar[r] & T\mathfrak{X}_{W} \end{tikzcd} \]
As the forgetful map $\mathfrak{X}^\Box_{W}\rightarrow \mathfrak{X}_{W}$ is formally smooth, it induces a surjection on the tangent spaces. Consequently it is sufficient to prove that the upper horizontal map is surjective.

Because of the commutative diagram $(\ref{diagramm1})$ this is equivalent to the surjectivity of the map 
\[ \pi_2:\;\sum_{\mathcal{F}\in\Tri(\mathcal{D})} T_{x_{\mathcal{F}[\frac{1}{t}]}}X_{K/\Qp,L} \longrightarrow T_{x} \gtilde_{K/\Qp,L}\]
induced by the second projection. Let $\alpha$ be the morphism of $K$-schemes $X= \gtilde\times_{\mathfrak{g}}\gtilde\rightarrow \gtilde$ given by the second projection. Let $\alpha_{K/\Qp}$ its image by the Weil restriction functor from $K$ to $\Qp$ and let $\alpha_{K/\Qp,L}$ be the base change of $\alpha_{K/\Qp}$ to $L$. For each $\tau\in\Sigma$ we write $\alpha_{\tau}$ for the base change of $\alpha$ by $\tau:K\rightarrow L$. Then we have the following decompositions
\begin{align*}
X_{K/\Qp}\times_{\Qp}L&\simeq\prod_{\tau\in\Sigma}X_{\tau}, \\
\gtilde_{K/\Qp,L}&\simeq\prod_{\tau\in\Sigma}\gtilde_\tau, \\
 \alpha_{K/\Qp,L}&=(\alpha_{\tau})_{\tau\in\Sigma}.
\end{align*}
Therefore it only remains to prove that for each $\tau\in\Sigma$, the $L$-linear map
\begin{equation}\label{differential_at_tau} d\alpha_{\tau}:\, \bigoplus_{\mathcal{F}\in\Tri(\mathcal{D})} T_{x_{\mathcal{F}[\frac{1}{t}],\tau}}X_\tau\longrightarrow T_{x_\tau}\gtilde_\tau \end{equation}
is surjective.

The $L\otimes_{\Qp}K_0$-linear endomorphism $\Phi\coloneqq \varphi^f$ of $D_{\cris}(\mathcal{D})$ induces an $L$-linear endomorphism $\Phi_{\tau}$ of $D_{\dR,\tau}(\mathcal{D})=D_{\cris,\tau|_{K_0}}(\mathcal{D})$ for all $\tau\in\Sigma$. This endomorphism is killed by the polynomial $\chi(D_{\cris}(\mathcal{D}),\varphi)\in L[X]$ which, by assumption, is separable and split. This implies that $\Phi_{\tau}$ is contained in a unique maximal split torus $T_{\tau}$ of $\GL(D_{\dR,\tau}(\mathcal{D}))$ or equivalently that the Zariski closure in $\GL(D_{\dR,\tau}(\mathcal{D}))$ of the group $\Phi_{\tau}^{\Z}$ is a maximal split torus $T_{\tau}$. If $\mathcal{F}$ is a triangulation of $\mathcal{D}$, the complete flag $D_{\cris}(\mathcal{F})$ of $D_{\cris}(\mathcal{D})$ is stable under $\varphi$, as is the complete flag $D_{\dR,\tau}(\mathcal{F}[\tfrac{1}{t}])$ under $\Phi_{\tau}$, and thus also $T_{\tau}$. However the maximal split torus $T_{\tau}$ fixes exactly $n!$ complete flags of $D_{\dR,\tau}(\mathcal{D})$. As $\mathcal{D}$ has exactly $n!$ triangulations we conclude that the set
\[ \{x_{\mathcal{F}[\frac{1}{t}],\tau},\, \mathcal{F}\in\Tri(\mathcal{D})\}\]
is exactly the set of points $(F,0,i^{-1}(\Fil_{\dR,\tau}))\in X_{\tau}(L)$ such that $F$ is fixed by the maximal split torus $T_{\tau}$.

The surjectivity of the map \eqref{differential_at_tau} is thus a direct consequence of Theorem \ref{intersectionBorels}. This concludes the first step of the proof.

Now consider the commutative diagram with exact lines and columns
\[
\begin{tikzcd}
0\ar[d]&0\ar[d]&\\
\bigoplus_{\mathcal{F}\in\Tri(\mathcal{D})} T\mathfrak{X}_{\mathcal{D}}^{\cris} \ar[r, "\sum"] \ar[d] &  T\mathfrak{X}_{\mathcal{D}}^{\cris} \ar[d,] \ar[r] & 0 \\
\bigoplus_{\mathcal{F}\in\Tri(\mathcal{D})} T\mathfrak{X}_{\mathcal{D},\mathcal{F}[\frac{1}{t}]} \ar[r, "\beta"] \ar[d]& T\mathfrak{X}_{\mathcal{D}} \ar[d,"W^+_{\dR}"] &  \\
\bigoplus_{\mathcal{F}\in\Tri(\mathcal{D})} T\mathfrak{X}_{W^{+}_{\dR}(\mathcal{D}),W^{+}_{\dR}(\mathcal{F}[\frac{1}{t}])} \ar[d] \ar[r] & T\mathfrak{X}_{W^{+}_{\dR}(\mathcal{D})} \ar[r]& 0 &\\
 0&&
\end{tikzcd}
\]
The exactness of the vertical lines are a consequences of Proposition \ref{kercris} and Corollary \ref{exactcristriangulin}. The surjectivity of $\sum$ is trivial and the surjectivity of the lower horizontal map is what we proved as a first step. We can deduce from this that the map $W_{\dR}^+\circ\beta$ is surjective and call upon the ``five'' Lemma to conclude that the map $\beta$ itself is surjective, which closes the proof of Theorem \ref{mainlocalthm}.
\end{proof}

\subsection{The case of Galois representations}

If $(\rho,V)$ is a continuous representation of the group $\mathcal{G}_K$ on some finite dimensional $L$-vector space $V$, let $\mathfrak{X}_{(\rho,V)}$ be the deformation functor over $\mathcal{C}$ of $(\rho,V)$. According to \cite{Bergereqdiff} there exists a functor $D_{\rig}$ from the category of continuous representation of the group $\mathcal{G}_K$ on finite dimensional $L$-vector spaces to the category of $(\varphi,\Gamma_K)$-modules over $\mathcal{R}_L$, which is proved fully faithful by \cite[Cor.~1.5]{ColmezDensite}. Then, \cite[Lem.~2.2.7]{BelChe} shows that if $A$ is an object of $\mathcal{C}$ and $(\rho,V)$ is a continuous representation of $\mathcal{G}_K$ with some $L$-algebras morphism $A\rightarrow\End_{\mathcal{G}_K}V$, then $V$ is a finite free $A$-module if and only if $D_{\rig}(V)$ is a finite free $\mathcal{R}_A$-module. The essential image of the functor $D_{\rig}$ is moreover stable under extensions. From these facts we conclude that if $(\rho,V)$ is a continuous representation of $\mathcal{G}_K$ on some finite dimensional $L$-vector space, then the functor $D_{\rig}$ induces an isomorphism of functors
\begin{equation}\label{identifydeforms} \mathfrak{X}_{(\rho,V)}\xrightarrow{D_{\rig}}\mathfrak{X}_{D_{\rig}(V)}. \end{equation}

If moreover $\mathcal{F}$ is a triangulation of $D_{\rig}(V)$ we define $\mathfrak{X}_{(\rho,V),\mathcal{F}[\frac{1}{t}]}$ as the functor from $\mathcal{C}$ to the category of sets sending an object $A$ to the set of isomorphism classes of tuples $(\rho_A,V_A,\pi_A,\mathcal{M}_A)$ where $(\rho_A,V_A)$ is a continuous representation of $\mathcal{G}_K$ on some finite free $A$-module $V_A$, $\pi_A$ is $\mathcal{G}_K$ equivariant $A$-linear map $V_A\rightarrow V$ inducing an isomorphism $L\otimes_AV_A\xrightarrow{\sim}V$ and $\mathcal{M}_A$ is a triangulation of $D_{\rig}(V_A)[\tfrac{1}{t}]$ such that $D_{\rig}(\pi_A)(\mathcal{M}_A)=\mathcal{F}[\tfrac{1}{t}]$.

Introducing the commutative diagram that, with the help of the functor $D_{\rig}$, relates the isomorphism $\beta$ of Theorem \ref{mainlocalthm} to the linear map
\[ \bigoplus_{\mathcal{F}\in\Tri(D_{\rig}(V))}T\mathfrak{X}_{(\rho,V),\mathcal{F}[\frac{1}{t}]}\longrightarrow T\mathfrak{X}_{(\rho,V)} \]
induced by the forgetful functors, we derive the following rephrasing of Theorem \ref{mainlocalthm}

\begin{coro}\label{mainlocalcoro}
Let $(\rho,V)$ be a $L$-linear $\varphi$-generic crystalline representation of the group $\mathcal{G}_K$ with regular Hodge-Tate type. The forgetful functors induce a surjective $L$-linear map
\[ \bigoplus_{\mathcal{F}\in\Tri(D_{\rig}(V))}T\mathfrak{X}_{(\rho,V),\mathcal{F}[\frac{1}{t}]}\longrightarrow T\mathfrak{X}_{(\rho,V)}. \]
\end{coro}

\subsection{Components of the non saturated deformation ring}\label{reminderqtriangulinecomponents}

This section contains some complements about the geometry of the formal scheme $\mathfrak{X}_{(\rho,V),\mathcal{F}[\frac{1}{t}]}$ that will be useful in the next chapter.

Let $(\rho,V)$ be some $n$-dimensional $L$-linear $\varphi$-generic crystalline representation of $\mathcal{G}_K$ with regular Hodge-Tate type. Let $\mathcal{D}=D_{\rig}(V)$ and let $\mathcal{F}$ be a triangulation of $\mathcal{D}$. We note $\mathcal{M}\coloneqq \mathcal{F}[\tfrac{1}{t}]$. Let $W_{\dR}(\mathcal{D})\coloneqq \BdR\otimes_{\BdR^+}W_{\dR}^+(\mathcal{D})$.

We fix a basis of $V$, i.e.~an $L$-linear isomorphism $\iota:\,L^n\simeq V$ so that $\rho$ can be identified with a group homomorphism $\rho:\,\mathcal{G}_K\rightarrow\GL_n(L)$. Let $\mathfrak{X}_\rho$ be the deformation functor of the pair $(\rho,\iota)$. Using the identification $(\ref{identifydeforms})$ we can define the deformation functors $\mathfrak{X}_{\rho,\mathcal{F}}$, $\mathfrak{X}_{\rho,\mathcal{M}}$, $\mathfrak{X}_{\rho}^{\cris}$. These are the obvious variants of the above functors in the context of $(\varphi,\Gamma_K)$-modules with the corresponding decorations.

Let $F\coloneqq D_{\cris}(\mathcal{F})$ the complete flag of $D_{\cris}(\mathcal{D})$ associated to the triangulation $\mathcal{F}$. For each $\tau\in\Sigma$, let $F_{\tau}$ be the complete flag of $D_{\dR,\tau}(D)$ image of $F$ under the functor $(-)\otimes_{L\otimes_{\Qp}K_0,\tau}L$. The stabilizer $B_{F,\tau}$ of $F_{\tau}$ is a Borel subgroup of $\GL(D_{\dR,\tau}(\mathcal{D}))$ and there exists a unique $w_{F,\tau}\in\mathfrak{S}_n$ such that $\Fil_{\dR,\tau}\in B_{F,\tau}w_{F,\tau}(F_{\tau})$. We define $w_{\mathcal{F}}\coloneqq (w_{F,\tau})_{\tau\in\Sigma}\in W=(\mathfrak{S}_n)^{[K:\Qp]}$.

It follows from \cite[Thm.~3.6.2.(ii)]{BHS3} and \cite[Prop.~3.6.4.]{BHS3} that the deformation functor $\mathfrak{X}_{\rho,\mathcal{M}}$ is pro-represented by some complete noetherian $L$-algebra $R_{\rho,\mathcal{M}}$ which is reduced, Cohen-Macaulay and equidimensional of dimension
\[ n^2+[K:\Qp]\frac{n(n+1)}{2}.\]
Its minimal primes are indexed by the set $\{ w\in W, \, w\geq w_{\mathcal{F}}\}$ where the ordering on $W$ is the Bruhat ordering. 

Let $R_{\rho,\mathcal{M}}^{w}$ be the quotient of $R_{\rho,\mathcal{M}}$ by the minimal prime with index $w$. By \cite[Thm.~3.6.2.(ii)]{BHS3}, it is Cohen-Macaulay and normal. 
Let $\mathfrak{X}_{(\rho,V),\mathcal{M}}^w$ be the subfunctor of $\mathfrak{X}_{(\rho,V),\mathcal{M}}$ defined as the image of
\begin{equation}\label{defRw} \Spf R_{\rho,\mathcal{M}}^{w}\subset \Spf R_{\rho,\mathcal{M}}\longrightarrow \mathfrak{X}_{\rho,\mathcal{M}}.\end{equation}
It can be easily checked that the inclusion $\mathfrak{X}_{(\rho,V),\mathcal{M}}^w\subset\mathfrak{X}_{(\rho,V),\mathcal{M}}$ is relatively representable. Note that the definition of $\mathfrak{X}_{\rho,\mathcal{M}}^w$ does not depend on the choice of the basis of the $L$-vector space $V$.

Let $\mathfrak{t}$ be the diagonal torus of $\mathfrak{g}=\mathfrak{gl}_{n,K}$. 
We recall that there is a canonical map $\tilde{\mathfrak{g}}\rightarrow \mathfrak{t}$ mapping $(A,gB)\in \tilde{\mathfrak{g}}$ to the class of $\mathrm{Ad}(g^{-1})A$ in $\mathfrak{b}/\mathfrak{u}$. Here $\mathfrak{u}\subset \mathfrak{b}$ is the sub-Lie-algebra of nilpotent upper triangular matrices, and the quotient $\mathfrak{b}/\mathfrak{u}$ is canonically identified with $\mathfrak{t}$.

This projection induces a canonical map $\Theta$ from $\mathfrak{X}_{(\rho,V),\mathcal{M}}$ to the completion at $(0,0)$ of the $L$-scheme $\mathfrak{t}_{K/\Qp,L}\times_{\mathfrak{t}_{K/\Qp,L}/W}\mathfrak{t}_{K/\Qp,L}$. The irreducible components of this scheme are in bijection with the group $W$. Let $\mathfrak{t}_w$ be the component defined by $\{(t,\mathrm{Ad}(w^{-1})t),\, t\in\mathfrak{t}_{K/\Qp,L}\}$ and $\widehat{\mathfrak{t}_w}$ its completion at $(0,0)$. We recall the ensuing characterization of $\mathfrak{X}_{(\rho,V),\mathcal{M}}^w$ which follows the precise definition of $\Theta$, as discussed in \cite[Cor.~3.5.12]{BHS3}.
\begin{prop}\label{mapdoubleweight}
Let $w_1$ and $w_2$ two elements of $\{w\in W,\, w\geq w_\mathcal{F}\}$. We have $\Theta(\mathfrak{X}_{(\rho,\mathcal{F}),\mathcal{M}}^{w_1})\subset\widehat{\mathfrak{t}_{w_2}}$ if and only if $w_1=w_2$.
\end{prop}

Finally we recall that the functor $\mathfrak{X}_{\rho}^{\cris}$ is pro-representable by a formally smooth $L$-algebra of dimension $n^2+[K:\Qp]\tfrac{n(n-1)}{2}$ (\cite{Kisindef}). It follow from the proof of \cite[Theorem 4.2.3]{BHS3} that we have
\[ \mathfrak{X}_{(\rho,V)}^{\cris}\subset \mathfrak{X}_{(\rho,V),\mathcal{M}}^{w_0} \]
as subfunctors of $\mathfrak{X}_{(\rho,V)}$ where, consistently with a well established notation introduced in Chapter \ref{intersection} above, $w_0$ stands for the longest element of $W$.

\section{Global deformation rings}

Let $F$ be a totally real field and $E$ a totally imaginary quadratic extension that we assume to be unramified over $F$ and such that all places $v$ dividing $p$ are split in $E$. Let $G$ be a unitary group in $n$ variables defined over $F$ such that $G\times_F E$ is an inner form of $\GL_{n,E}$. We assume moreover that $G(F\otimes_{\Q}\R)$ is compact and that the group $G$ is quasi-split over all finite places of $F$. This implies that $n$ is odd or that $4|n[F:\Q]$. If $v$ is a place of $F$ which splits in $E$, the group $G$ splits at $v$. We fix a place $\vtilde$ of $E$ dividing $v$ and an isomorphism $G\times_FE_{\vtilde}\cong \GL_{n,E_{\vtilde}}$ which induces an isomorphism $G\times_F F_v\simeq\GL_{n,F_v}$. Let $B_v\subset G(F_v)$ be the subgroup corresponding to the Borel subgroup of upper triangular matrices of $\GL_n(F_v)$ under this isomorphism and $T_v\subset B_v$ the subgroup corresponding to the subgroup of diagonal matrices in $\GL_n(F_v)$. We write $T=\prod_{v|p}T_v$ and $B_p=\prod_{v|p}B_v$. Moreover we define $U_v\subset G(F_v)$ the maximal compact subgroup of $G(F_v)$ corresponding to $\GL_n(\mathcal{O}_{F_v})$ under this isomorphism.

Let $U^p$ be a compact open subgroup of $G(\mathbb{A}^{p,\infty})$ of the form $\prod_{v\nmid p} U_v$ with $U_v$ a compact open subgroup of $G(F_v)$ which is assumed to be hyperspecial when $v$ is a place of $F$ which is inert in $E$. 
Let $S_p$ denote the set of places of $F$ that divide $p$ and let $S$ be a finite set of places of $F$ containing $S_p$ and the finite set of places of $F$ for which $U_v$ is not hyperspecial.
Finally we write $U=U^p\times U_p$, where $U_p=\prod_{v|p}U_v$ is a maximal compact subgroup of $G(F\otimes_{\Q}\Q_p)$.

We write $E_S$ for the maximal extension of $E$ that is unramified outside all places of $E$ above the places in $S$ and denote by $\mathcal{G}_{E,S}=\mathrm{Gal}(E_S/E)$ the corresponding Galois group. Let $A$ be a $\Zp$-algebra and $\rho_A$ a representation of $\mathcal{G}_{E,S}$ on some finite free $A$-module $V_A$ of rank $n$. We write $\rho_A^c$ for the representation $g\mapsto \rho_A(cgc)$, where $c\in \mathrm{Gal}(\overline{F}/F)$ is a complex conjugation. The representation $(\rho_A,V_A)$ is called \emph{polarizable}, if there exists an isomorphism
\[({\rho}^{\vee,c},V_A^\vee)\cong({\rho}\otimes\epsilon^{n-1},V_A),\]
where $\epsilon$ is the cyclotomic character. Such an isomorphism is called \emph{polarization}.

We fix $L$ a finite extension of $\Qp$ and $(\overline{\rho},\overline{V})$ a continuous polarized representation $\mathcal{G}_{E,S}\rightarrow\GL_n(k_L)$ which is absolutely irreducible so that it has a unique polarization up to scalar multiplication. We denote by $R_{\overline{\rho},S}$ the universal polarized deformation $\mathcal{O}_L$-algebra of $\overline{\rho}$. 
That is, the complete local $\mathcal{O}_L$-algebra pro-representing the functor of isomorphism classes of triples $(\rho_A,V_A,\iota_A)$, with $V_A$ a finite free $A$-module with a continuous polarized action $\rho_A$ of $\mathcal{G}_{E,S}$ and an isomorphism $\iota_A:V_A/\mathfrak{m}_AV_A\cong \overline{V}$ of $\mathcal{G}_{E,S}$-representations, on the category of local Artinian $\mathcal{O}_L$-algebras $A$ with residue field $k_L$. The existence of the $\mathcal{O}_L$-algebra $R_{\overline{\rho},S}$ follows from \cite[\S1.1]{Chefougere}.

Let $\mathcal{X}_{\overline{\rho},S}=(\Spf R_{\overline{\rho},S})^{\rig}$ be the rigid analytic generic fiber of the formal scheme $\Spf R_{\overline{\rho},S}$. As $(\overline{\rho},\overline{V})$ is absolutely irreducible, the $L$-points of $\mathcal{X}_{\overline{\rho},S}$ are in bijection with the set of isomorphism classes of continuous representations $(\rho,V)$ of $\mathcal{G}_{E,S}$ on $L$-vector spaces such that $\rho^{\vee,c}\simeq\rho\otimes \varepsilon^{n-1}$ and such that there exists a $\mathcal{G}_{E,S}$-stable $\mathcal{O}_L$-lattice $V^\circ\subset V$ and a $\mathcal{G}_{E,S}$-equivariant isomorphism $V^\circ/\varpi_L V^\circ\simeq\overline{V}$. Given a point $x\in\mathcal{X}_{\overline{\rho}}$, we denote by $(\rho_x,V_x)$ the associated representation of $\mathcal{G}_{E,S}$.

Fix an isomorphism $\iota:\overline{\Qp}\simeq\C$. Recall that, if $\pi$ is a (cuspidal) automorphic representation of $G$, there exists a unique polarized $n$-dimensional $\overline{\Qp}$-representation $(\rho_{\pi},V_{\pi})$ of $\mathrm{Gal}(\overline{E}/E)$ associated to $\pi$. If $(\pi^{p,\infty})^{U^p}\neq 0$ then this representation factors through $\mathcal{G}_{E,S}$. The existence of this Galois representation is a consequence of base change (\cite[Cor.~5.3]{Labesse}) and of the construction of Galois representations associated to some automorphic representation of $\GL_{n,E}$ (see \cite{CheHaII}). We say that a point $x\in\mathcal{X}_{\overline{\rho},S}(L)$ is $(G,U^p)$-automorphic (resp.~$(G,U)$-automorphic) if there exists a (cuspidal) automorphic representation $\pi$ of $G$ such that $(\pi^{p,\infty})^{U^p}\neq 0$ (resp.~such that~$(\pi^\infty)^U\neq 0$) and such that there is an isomorphism $(\rho_x,V_x\otimes_L\overline{\Qp})\simeq(\rho_{\pi},V_{\pi})$. Moreover we say that $(\overline{\rho},\overline{V})$ is $(G,U)$-automorphic over $L$ if there exists a $(G,U)$-automorphic point $x\in\mathcal{X}_{\overline{\rho},S}(L)$. Let $\mathcal{X}_{\overline{\rho},S}^{\aut}$ be the Zariski closure of the set of $(G,U)$-automorphic points in $\mathcal{X}_{\overline{\rho},S}^{\aut}$. 
The aim of this section is to prove the following theorem:

\begin{theo}\label{globalmaintheorem}
Assume that $p>2$, that all places of $S$ are split in $E$ and that the group $\overline{\rho}(\mathcal{G}_{E(\zeta_p)})$ is adequate in the sense of \cite[Definition 2.3]{ThorneAut}. Then the inclusion $\mathcal{X}_{\overline{\rho},S}^{\aut}\subset\mathcal{X}_{\overline{\rho},S}$ is the inclusion of a union of irreducible components (possibly empty if $(\overline{\rho},\overline{V})$ is not $(G,U)$-automorphic). 
\end{theo}

From now on we assume $p>2$, the places of $S$ split in $E$, $\overline{\rho}(\mathcal{G}_{E(\zeta_p)})$ adequate and that $(\overline{\rho}, \overline{V})$ is $(G,U)$-automorphic.

Recall that, for a place $v\in S$, we fix a place $\tilde v$ of $E$ dividing $v$. We write $\mathcal{G}_{E_{\tilde{v}}}$ for the choice of a decomposition group at $\tilde{v}$. Given a representation $\rho$ of $\mathcal{G}_{E,S}$ we write $\rho_{\tilde{v}}$ for the restriction of $\rho$ to $\mathcal{G}_{E_{\tilde{v}}}$. 

Finally we assume that $U^p$ is sufficiently small so that the compact open sub group $U\coloneqq \prod_{v}U_v$ is such that
\begin{equation}\label{net} \forall g\in G(\mathbb{A}_F^\infty),\, G(F)\cap g Ug^{-1}=\{1\}.\end{equation}

\subsection{Recollections about eigenvarieties and patching}
Attached to the data $G, U^p$ and $\overline{\rho}$ there is a so-called \emph{eigenvariety}. For a place $v$ dividing $p$ let us write $\hat T_v$ for the rigid analytic space of continuous characters of $T_v$ and similarly $\hat T^0_v$ for the space of continuous characters of the maximal compact subgroup $T^0_v\subset T_v$. Further let 
\[\hat T=\prod_{v|p} \hat T_v \hspace{1cm} \text{and}\hspace{1cm} \hat T^0=\prod_{v|p} \hat T_v^0.\]
The eigenvariety associated to $G, U^p$ and $\overline{\rho}$ is by definition the Zariski-closed rigid analytic subspace $Y(U^p,\overline{\rho})\subset \mathcal{X}_{\overline{\rho},S}\times \hat T_L$ that is the (scheme-theoretic) support of the locally analytic Jacquet-module $J_{B_p}( \hat S(U^p,L)^{\an}_\mathfrak{m})$ of the locally analytic representation underlying the $G(F\otimes_\Q\Q_p)$-representation on the space $\hat S(U^p,L)_{\mathfrak{m}}$ of $p$-adic automorphic forms of tame level $U^p$. 
Here $\mathfrak{m}$ is a certain maximal ideal of a Hecke-algebra corresponding to the residual Galois representation $\overline{\rho}$.
We refer to \cite[3.1]{BHS2} for details of this construction. 

We recall the notion of a \emph{classical point} on $Y(U^p,\overline{\rho})$:
We write $X^{\ast}(T)$ for the space of algebraic characters of the product of the diagonal tori in 
\begin{equation}\label{splitWeilrestriction}
(\Res_{F/\Q}G)_{\C}\cong \prod_{\tau:F\hookrightarrow \C} \GL_{n,\C}.
\end{equation}
This space comes equipped with an action of the Weyl group $W$ of $(\Res_{F/\Q}G)_{\C}$. As usual we write $w\cdot \lambda$ for the shifted \emph{dot}-action of $W$ on $X^\ast(T)$. We write $w_0$ for the longest element of $W$.

The isomorphism $\overline{\Q_p}\cong \C$ identifies $\lambda\in X^\ast(T)$ with a character $T_p\rightarrow \overline{\Q_p}^\times$ that we denote by $z^\lambda$. If $L$ is a finite extension of $\Q_p$ such that $\Res_{F/\Q}G$ splits over $L$, then $z^\lambda$ takes values in $L$ and we may view it as an $L$-valued point of $\hat T$.

Given a representation $\pi_\infty$ of $G(F\otimes_\Q\R)$ we say that $\pi_\infty$ is of weight $\lambda\in X^\ast(T)$ if it is the restriction to $G(F\otimes_\Q\R)$ of the irreducible algebraic representation of $(\Res_{F/\Q}G)_{\C}$ of highest weight $\lambda$.

Let $\pi=\pi_\infty\otimes_\C\pi^{p,\infty}\otimes_\C\pi_p$ be an automorphic representation of $G$ such that $(\pi^{p,\infty})^{U^p}\neq 0$ and such that $\pi_\infty$ is of weight $\lambda$. 
Moreover we assume that, for all $v\in S_p$, the representation $\pi_v$ is an unramified quotient of the smooth induced representation $(\Ind_{B_v}^{G_v}\delta_{\sm,v}\delta_v)^{\sm}$ for some unramified character $\delta_{\sm,v}$ of $T_v$ with values in $L^\times$ and where $\delta_v$ is the smooth character
\[ \delta_v=(1\otimes |\cdot|_v\otimes\cdots\otimes |\cdot|_v^{n-1}).\]
Let $\delta_{\sm}\coloneqq\bigotimes_{v\in S_p}\delta_{\sm,v}$. The associated Galois-representation $\rho_\pi$ is $(G,U)$-automorphic by definition and we have \[(\rho_{\pi},\delta_{\sm} z^\lambda)\in Y(U^p,\overline{\rho})(L)\subset\mathcal{X}_{\overline{\rho},S}(L)\times \hat T(L),\]
see \cite[Proposition 3.4]{BHS2} for example. The point $x=(\rho_{\pi},\delta_{\sm} z^\lambda)$ is called the classical point associated with $(\pi,\delta_{\sm})$.

It follows from \cite{CheHaII} that, for $v\in S_p$, the representation $\rho_{\vtilde}$ is crystalline and that the character $\delta_{\sm,v}$ is of the form $\delta_{\mathcal{F}_{\vtilde}}$ for $\mathcal{F}_{\vtilde}$ a triangulation of $\rho_{\vtilde}$ (see \S\ref{crystalline} for the definition of $\delta_{\mathcal{F}_{\vtilde}}$). We say that $\rho$ is crystalline $\varphi$-generic if $\rho_{\vtilde}$ is crystalline $\varphi$-generic for all places $v$ dividing $p$. 

Assuming that $\rho$ is crystalline $\varphi$-generic, it follows from the classification of intwertinning operators between principal series that
\[ \mathcal{F}_v \longmapsto \delta_{\mathcal{F}_v}\]
induces a bijection between the set of smooth characters
$\delta_{\sm}$ such that $\pi_v\cong (\Ind_{B_v}^{G(F_v)}\delta_{\sm}\delta_v)^{\sm}$ and the triangulations of $\rho_{\vtilde}$.

Similarly, given a tuple $\underline{\mathcal{F}}=(\mathcal{F}_v)_{v\in S_p}$ of refinements we write $\delta_{\underline{\mathcal{F}}}=(\delta_{\mathcal{F}_v})_v$ for the corresponding unramified character of $T_p$. In this case $x_{\underline{\mathcal{F}}}\coloneqq (\rho,z^{\lambda}\delta_{\underline{\mathcal{F}}})$ is a classical point of $Y(U^p,\overline{\rho})$ by construction, associated to the pair $(\pi,\delta_{\underline{\mathcal{F}}})$. 

Fix an embedding $\tau:F_v\hookrightarrow \bar\Q_p$. Via the identification $\overline{\Q_p}\cong \C$ this embedding defines an embedding $\Q\hookrightarrow \C$ and we write $W_\tau$ for the factor of the Weyl group $W$ corresponding to this embedding via the decomposition $(\ref{splitWeilrestriction})$.

The relative position of the $\tau$-part of the Hodge Filtration 
\begin{align*}
\Fil_{\dR,\tau}\coloneqq \Fil_{\dR}\otimes_{F_v\otimes_{\Qp} L, \tau\otimes{\Id}}\bar\Q_p&\subset  D_{\dR}(\rho_v)\otimes_{F_v\otimes_{\Qp} L, \tau\otimes{\Id}}\bar\Q_p\\ 
&=D_{\cris}(\rho_v)\otimes_{F_{v,0}\otimes_{\Qp} L, \tau|_{F_{v,0}}\otimes{\Id}}\bar\Q_p
\end{align*}
with respect to $\mathcal{F}_v\otimes_{F_{v,0}\otimes_{\Qp} L,\tau\otimes{\Id}}\bar\Q_p$ defines an element of the Weyl group $w_{{\mathcal{F}}_\tau}\in W_\tau$. We write $w_{\underline{\mathcal{F}}}\in W$ for the Weyl group element defined by the tuple $\underline{\mathcal{F}}$.

The following proposition summarizes the properties of the eigenvariety needed for the proof of the main theorem:
\begin{prop}\label{propertiesofeigenvar}
\begin{enumerate}[(i)]
\item The eigenvariety $Y(U^p,\overline{\rho})$ is reduced and equi-dimensional of dimension\[\dim Y(U^p,\overline{\rho})=\dim \hat T^0=n[F:\Q].\]
\item The set of classical points as defined above is Zariski-dense and has the accumulation property, i.e.~for every classical point $x$ and every open connected neighborhood $U$ of $x$ the classical, crystalline $\varphi$-generic points are Zariski-dense in $U$.
\item\label{companion} Let $x=x_{\underline{\mathcal{F}}}$ be a classical, crystalline $\varphi$-generic point associated to $(\pi,\delta_{\underline{\mathcal{F}}})$ as above. For a weight $\mu\in X^*(T)$, one has
\[(\rho,z^\mu\,\delta_{\underline{\mathcal{F}}})\in Y(U^p,\overline{\rho})\Longleftrightarrow \mu=ww_0\cdot\lambda\ \text{with}\ w\in W, \ w_{\underline{\mathcal{F}}}\preceq w.\]
For $w_{\underline{\mathcal{F}}}\preceq w$, we define
\[ x_{\underline{\mathcal{F}},w}\coloneqq (\rho,z^{ww_0\cdot\lambda}\delta_{\underline{\mathcal{F}}}).\]
\item Let $x$ be as in (\ref{companion}). Then the projection $\omega:Y(U^p,\overline{\rho})\rightarrow \hat T^0$ is flat at the points $x_{\underline{\mathcal{F}},w}$.
\item The projection $Y(U^p,\overline{\rho})\rightarrow \mathcal{X}_{\overline{\rho},S}$ is locally on the source and the target a finite morphism.
\end{enumerate}
\end{prop}
\begin{proof}
\noindent Points (i) and (ii) are contained in \cite[3.8]{CheJL}. The statements can be obtained as well along the lines of Corollaire 3.12, Th\'eor\`eme 3.19 and Corollaire 3.20 of \cite{BHS1}. See Definition 3.2 and Proposition 3.4 of \cite{BHS2} for a comparison of the (a priori different) notions of classical points. \\
\noindent (iii) This is a direct consequence of \cite[Theorem 5.3.3]{BHS3}.\\
\noindent (iv) This is contained in \cite[Theorem 5.4.2]{BHS3}.\\
\noindent (v) The map $Y(U^p,\overline{\rho})\rightarrow\mathcal{X}_{\overline{\rho},S}$ is the composite of the closed embedding $Y(U^p,\overline{\rho})\subset \mathcal{X}_{\overline{\rho},S}\times \hat T$ with the projection $\mathcal{X}_{\overline{\rho},S}\times \hat T\rightarrow \mathcal{X}_{\overline{\rho},S}$. These two maps are locally of finite type so that the map $Y(U^p,\overline{\rho})\rightarrow\mathcal{X}_{\overline{\rho}}$ is locally of finite type. 
We claim that the fibers are discrete and hence the morphism is locally quasi-finite. The Proposition then follows from \cite[Prop.~1.5.4.(c)]{Huber}.

Indeed, consider the morphism $ \mathcal{X}_{\overline{\rho},S}\rightarrow \mathbb{A}^{n[F:\Q]}/W$ given by mapping $\rho$ to the set of Hodge-Tate weights of the $\rho_v$, for $v|p$. 
Fix a point $\rho$ and write $\mathrm{HT}(\rho)$ for its image in $ \mathbb{A}^{n[F:\Q]}/W$ for the moment. 
The composition \[q:\hat T^0\longrightarrow \mathbb{A}^{n[F:\Q]}\longrightarrow \mathbb{A}^{n[F:\Q]}/W\] of the logarithm with the projection map is obviously quasi-finite and hence $q^{-1}(\mathrm{HT}(\rho))$ is a discrete set. 
Finally the weight map $Y(U^p,\overline{\rho})\rightarrow \hat T^0$ is quasi-finite by the usual argument using special coverings of Fredholm hypersurfaces (see e.g.~\cite[Proposition 3.11]{BHS1}). And hence the preimage of $q^{-1}(\mathrm{HT}(\rho))$ under the weight map is still a discrete set. As this set contains the fiber of $Y(U^p,\overline{\rho})\rightarrow\mathcal{X}_{\overline{\rho},S}$ over $\rho$ the claim follows.
\end{proof}

We further recall the \emph{patched eigenvariety} $X_p(\overline{\rho})$ and its relation to the global object $Y(U^p,\overline{\rho})$. In \cite[3]{BHS1} we have carried out the following construction:
Let us write 
\[R_{\overline{\rho}_p}=\widehat{\bigotimes}_{v\in S_p} R_{\overline{\rho}_{\vtilde}}\ \text{and}\ R_{\overline{\rho}^p}=\widehat{\bigotimes}_{v\in S\backslash S_p} R_{\overline{\rho}_{\vtilde}}\]
for the completed tensor products of the maximal reduced and $\Z_p$-flat quotients $R_{\overline{\rho}_v}$ of the universal framed deformation rings $R_{\overline{\rho}_{\vtilde}}'$ of $\overline{\rho}_{\vtilde}$. Let
\[ R_{\overline{\rho},\mathcal{S}}\coloneqq R_{\overline{\rho},S}\otimes_{\left(\widehat{\bigotimes}_{v\in S}R'_{\overline{\rho}_{\vtilde}}\right)}\left(\widehat{\bigotimes}_{v\in S}R_{\overline{\rho}_{\vtilde}}\right).\]
There exists an integer $g\geq 1$ and a commutative diagram with maps of local $\mathcal{O}_L$-algebras
\begin{equation}\label{diagpatching} \begin{tikzcd} S_\infty\coloneqq \mathcal{O}_L\dbl \Zp^q \dbr \ar[r] \ar[d] &
R_\infty\coloneqq \big(R_{\overline{\rho}_p}\widehat{\otimes}_{\mathcal{O}_L}R_{\overline{\rho}^p}\big)\dbl y_1,\dots,y_g\dbr \ar[d, twoheadrightarrow] \\
R_{\infty}\otimes_{S_\infty} \mathcal{O}_L \ar[r,twoheadrightarrow] & R_{\overline{\rho},\mathcal{S}} \end{tikzcd}
\end{equation}
where the left vertical map is induced by the augmentation map $S_\infty\rightarrow \mathcal{O}_L$ and where $q=g+[F:\Q]\frac{n(n-1)}{2}+n^2 |S|$.

We write $\mathcal{X}_\infty$ and $\mathcal{X}_{\overline{\rho}_p}$ for the rigid analytic generic fibers of $\Spf R_\infty$ and $\Spf R_{\overline{\rho}_p}$. Moreover we denote by $X_p(\overline{\rho})\subset \mathcal{X}_\infty\times\hat T$ the patched eigenvariety constructed in \cite{BHS1}. Then there is a canonical embedding
\begin{equation}\label{compareisonoflocalandpatchedeigenvar}
Y(U^p,\overline{\rho})\hookrightarrow X_p(\overline{\rho})\times_{(\Spf S_\infty)^{\rig}}\Sp L\subset \mathcal{X}_\infty\times \hat T,
\end{equation}
see \cite[4.1]{BHS1} or \cite[(5.34)]{BHS3}. Let us abbreviate $X_p(\overline{\rho})\times_{(\Spf S_\infty)^{\rig}}\Sp L$ by $Y_p(\overline{\rho})$ for the moment.
The precise relation of the local geometry of the patched eigenvariety and the (global) eigenvariety is given by the following proposition:
\begin{prop}\label{propcomparecomparelocalrings}
Assume that $x=x_{\underline{\mathcal{F}}}$ is a classical, crystalline $\varphi$-generic point. For each $w\in W$ such that $x_{\underline{\mathcal{F}},w}\in Y(U^p,\overline{\rho})$ the morphism $(\ref{compareisonoflocalandpatchedeigenvar})$ induces an isomorphism of complete local rings
\[\hat{\mathcal{O}}_{Y_p(\overline{\rho}),x_{\underline{\mathcal{F}},w}}\cong \hat{\mathcal{O}}_{Y(U^p,\overline{\rho}),x_{\mathcal{F},w}}.\]
\end{prop}
\begin{proof}
This is \cite[Proposition 5.4.1]{BHS3}.
\end{proof}
Finally we recall the relation of the patched eigenvariety with the space of trianguline representations, see \cite{BHS1}. Let $X_{\tri}(\overline{\rho})=\prod_{v\in S_p}X_{\tri}(\overline{\rho}_v)\subset \mathcal{X}_{\overline{\rho}_p}\times \hat T$. Then there is a commutative diagram
\begin{equation}\label{patchedeigenvarvsXtri}
\begin{tikzcd}
X_p(\overline{\rho}) \ar[r, "\iota"]\ar[d] & X_{\tri}(\overline{\rho})\ar[d] \times \mathcal{X}_{\overline{\rho}^p}\times \mathbb{U}^g\\
\hat T^0 \ar[r, "\cong"]& \hat T^0,
\end{tikzcd}
\end{equation}
where $\iota$ is a closed embedding that identifies $X_p(\overline{\rho})$ with a union of irreducible components of the target. Here $\mathbb{U}^g=(\Spf \mathcal{O}_L\dbl y_1,\dots, y_g\dbr)^{\rig}$, see \cite[Theorem 3.21]{BHS1}.

\subsection{A characterization of the tangent space}\label{subsectionglobaltangentspace}
We fix a $(G,U)$-automorphic representation $\rho\in \mathcal{X}_{\overline{\rho},S}\subset \mathcal{X}_\infty$ that is crystalline $\varphi$-generic. For the reminder of this subsection we introduce the following notations:

Let $R$ be the complete local ring of $\mathcal{X}_\infty$ at $\rho$ so that $(\mathcal{X}_\infty)_\rho^{\hat{}}=\Spf R$ and, for a given refinement $\underline{\mathcal{F}}=(\mathcal{F}_v)_{v\in S_p}$ of $\rho$ and $w\in W$ such that $w_{\underline{\mathcal{F}}}\preceq w$ let $R_{\underline{\mathcal{F}},w}$ be the complete local ring of $X_p(\overline{\rho})$ at the point $x_{\underline{\mathcal{F}},w}$, so that $X_p(\overline{\rho})^{\hat{}}_{x_{\underline{\mathcal{F}},w}}=\Spf R_{\underline{\mathcal{F}},w}$.  By \cite[Lemma 4.3.3]{BHS3},  the canonical map $R\rightarrow R_{\underline{\mathcal{F}},w}$ is a surjection.
Similarly we define $S$ as the complete local ring of $\mathcal{X}_{\overline{\rho},S}$ at $\rho$ and $S_{\underline{\mathcal{F}},w}$ the complete local ring of $Y_p(\overline{\rho})$ at $x_{\underline{\mathcal{F}},w}$. Then we have a canonical surjection $R\twoheadrightarrow S$ and, by Proposition \ref{propcomparecomparelocalrings} an identification $S_{\underline{\mathcal{F}},w}=S\otimes_R R_{\underline{\mathcal{F}},w}=R_{\underline{\mathcal{F}},w}\otimes_{S_\infty}\mathcal{O}_L$.

Obviously the ring $R_{\underline{\mathcal{F}},w}$ decomposes as a tensor product 
\[R_{\underline{\mathcal{F}},w}=\widehat\bigotimes_{v|p}R_{{\mathcal{F}_v,w_v}}\hat\otimes \widehat{\mathcal{O}}_{\mathcal{X}_{{\overline{\rho}^p}\times\mathbb{U}^g},{\rho^p}},\] 
where $R_{{\mathcal{F}_v,w_v}}$ is the complete local ring of $X_{\tri}(\overline{\rho}_v)$ at $x_{\mathcal{F}_v,w_v}=(\rho_v,z^{w_vw_0\cdot\lambda_v}\delta_{\mathcal{F}_v})$; and where $\rho^p$ is the image of $\rho$ in $\mathcal{X}_{\overline{\rho}^p}\times\mathbb{U}^g$.

By \cite[Cor.~3.7.8]{BHS3} and \cite[Thm.~3.6.2.(ii)]{BHS3} the quotient $R_{{\mathcal{F}_v,w_v}}$ of $R_{\overline{\rho}_{\vtilde}}$ coincides with the quotient $R_{\overline{\rho}_{\vtilde},\mathcal{F}_v}^{w_v}$ of $R_{\overline{\rho}_{\vtilde},\mathcal{F}_v}$ defined in section \ref{reminderqtriangulinecomponents}. 
We conclude that the map $R_{\infty}\rightarrow R_{\underline{\mathcal{F}},w}$ induces an isomorphism
\begin{equation}\label{comparisontriangulinedeformation} R_{\infty}\otimes_{\left(\widehat{\bigotimes}_{v|p}R_{\overline{\rho}_{\vtilde}}\right)}\left(\widehat{\bigotimes}_{v|p}R_{\overline{\rho}_{\vtilde},\mathcal{F}_v}^{w_v}\right) \xrightarrow{\sim} R_{\underline{\mathcal{F}},w}.\end{equation}

Let us write $\Spec R_{\underline{\mathcal{F}}}$ for the scheme theoretic image of the canonical morphism
\[\coprod_{w_{\underline{\mathcal{F}}}\preceq w} \Spec R_{\underline{\mathcal{F}},w}\longrightarrow \Spec R\]
and $\Spec S_{\underline{\mathcal{F}}}$
for the scheme theoretic image of the canonical morphism
\[\coprod_{w_{\underline{\mathcal{F}}}\preceq w} \Spec S_{\underline{\mathcal{F}},w}\longrightarrow \Spec S.\]
\begin{lemm}\label{lemmagenericpoints}
\begin{enumerate}[(i)]
\item The scheme $\Spec R_{\underline{\mathcal{F}}}$ is reduced and Cohen-Macaulay of dimension 
\[g+[F:\Q]\frac{n(n+1)}{2}+n^2|S|=q+n[F:\Q].\]
\item The scheme $\Spec S_{\underline{\mathcal{F}},w}$ is reduced and equi-dimensional of dimension $n[F:\Q]$. The same holds true for $\Spec S_{\underline{\mathcal{F}}}$\\
\item Let $\eta\in \Spec S_{\underline{\mathcal{F}},w}$ be a generic point, then $\eta\notin \Spec R_{\underline{\mathcal{F}},w'}$ for $w'\neq w$.
\end{enumerate}
\end{lemm}
\begin{proof}
\noindent (i) As $\rho$ is automorphic the space $\mathcal{X}_{\overline{\rho}^p}$ is smooth of dimension $n^2|S\backslash S_p|$ at (the image of) $\rho$, by \cite[Theorem 1.2]{Caraiani} and \cite[Lemma 1.3.2 (1)]{BLGGT}.
The claim now follows from diagram $(\ref{patchedeigenvarvsXtri})$, isomorphism \eqref{comparisontriangulinedeformation} and the fact that $R_{\rho_{\vtilde},\mathcal{F}_v}^{w_v}$ is reduced, normal and Cohen-Macaulay of dimension $n^2+[F_v:\Q_p]\tfrac{n(n+1)}{2}$ (see section \ref{reminderqtriangulinecomponents}).\\
\noindent (ii) As $S_{\underline{\mathcal{F}},w}$ is the complete local ring at some point of the eigenvariety $Y(U^p,\overline{\rho})$, the claims follow from the corresponding statements on $Y(U^p,\overline{\rho})$ in Proposition \ref{propertiesofeigenvar}. The claim on $\Spec S_{\underline{\mathcal{F}}}$ then is a direct consequence. \\ 
\noindent (iii) 
Let us write $\Spf A$ for the formal completion of $\hat T^0$ at $\omega(\iota(x))$. Thus $A$ is just the completed tensor-product of power series rings of dimension $n[F_v:\Q_p]$ indexed by $v\in S_p$. 
We identify $\Spf A$ with the formal completion of $ {\mathfrak{t}}_{K/\Q_p,L}$ at the origin via $\kappa\mapsto \kappa-\omega(\iota(x))$.

Taking products over all $v\in S_p$ and the product with the formally smooth contribution from $\mathcal{X}_{\overline{\rho}^p}\times \mathbb{U}^g$ we obtain a commutative diagram as in \cite[2.5]{BHS3}:
\[\begin{tikzcd}
\Spf S_{\underline{\mathcal{F}},w} \ar[r, hookrightarrow]\ar[rd]& \Spf R_{\underline{\mathcal{F}},w} \ar[d,]\ar[r, hookrightarrow]& \Spf R_{\underline{\mathcal{F}}}\ar[d, "\underline{\Theta}"]\\
& \Spf A \ar[r,"\psi"]& \prod_{v|p} (\mathfrak{t}_{F_v/\Q_p,L}\times_{\mathfrak{t}_{F_v/\Q_p,L}/W_v}\mathfrak{t}_{F_v/\Q_p,L} )^{\hat {}}_0,
\end{tikzcd}\]
where $\underline{\Theta}$ is the product of all the maps $\Theta$ for all $v|p$, as defined before Proposition \ref{mapdoubleweight}.

Recall that $\mathfrak{t}_{F_v/\Q_p,L}\times_{\mathfrak{t}_{F_v/\Q_p,L}/W_v}\mathfrak{t}_{F_v/\Q_p,L}$ decomposes as a product $\prod_{w'_v\in W_v}\mathfrak{t}_{w'_v}$. It follows from \cite[Cor.~3.5.12]{BHS3} that the morphism $\psi$ identifies $\Spf A$ with 
the product $\hat {\mathfrak{t}}_{w}=\prod_{v|p}\hat{\mathfrak{t}}_{w_v,0}$.
Let us write $\Spf B$ for the formal scheme in the lower left corner of the diagram, and $\Spf B_{w'}=\hat {\mathfrak{t}}_{w'}=\prod_{v|p}\hat{\mathfrak{t}}_{w'_v,0}$.

Passing to rings and applying $\Spec(-)$ the above diagram becomes:

\[\begin{tikzcd}
\Spec S_{\underline{\mathcal{F}},w} \ar[r, hookrightarrow]\ar[rd, "\omega"]& \Spec R_{\underline{\mathcal{F}},w} \ar[d,]\ar[rr, hookrightarrow]&& \Spec R_{\underline{\mathcal{F}}}\ar[d, "\underline{\Theta}"]\\
& \Spec A \ar[r,"\cong"]&\Spec B_w\ar[r, hookrightarrow]&\Spec B,
\end{tikzcd}\]
where we use the same latters for the maps by abuse of notation.
Now, by Proposition \ref{propertiesofeigenvar}, the map $\omega$ is flat and hence dominant when restricted to each irreducible component of $\Spec S_{\underline{\mathcal{F}},w}$.
On the other hand $(\underline{\Theta})(\Spec R_{\underline{\mathcal{F}},w'})=\Spec B_{w'}$ by \cite[Cor.~3.5.12]{BHS3} (compare also Proposition \ref{mapdoubleweight}), and $\Spec B_{w'}$ does not contain the generic point of $\Spec B_w$ for $w\neq w'$. The claim follows from this.
\end{proof}
\begin{lemm}\label{lemmaqequations}
The canonical maps $R_{\underline{\mathcal{F}}}\otimes_{S_\infty}\mathcal{O}_L\rightarrow R_{\underline{\mathcal{F}}}\otimes_R S$ and $R_{\underline{\mathcal{F}}}\otimes_R S \rightarrow S_{\underline{\mathcal{F}}}$ are isomorphisms, i.e.
\begin{equation}\label{intersectionasschemes}
\Spec S_{\underline{\mathcal{F}}}=\Spec R_{\underline{\mathcal{F}}}\cap \Spec S=\Spec R_{\underline{\mathcal{F}}}\times_{\Spec R} \Spec S
\end{equation}
as subschemes of $\Spec R$. In particular $\Spec S_{\underline{\mathcal{F}}}\subset\Spec R_{\underline{\mathcal{F}}}$ is a closed subscheme that is cut out by $q$ equations.
\end{lemm}
\begin{proof}
From \eqref{diagpatching}, we deduce a sequence of surjective maps
\[ R_{\underline{\mathcal{F}}}\otimes_{S_\infty}\mathcal{O}_L\twoheadrightarrow R_{\underline{\mathcal{F}}}\otimes_RS\twoheadrightarrow S_{\mathcal{F}}.\]

First note that  $\Spec R_{\underline{\mathcal{F}}}\times_{\Spec S_\infty} \Spec \mathcal{O}_L$ is cut out by $q$ equations in $\Spec R_{\underline{\mathcal{F}}}$.

By definition $\Spec R_{\underline{\mathcal{F}}}=\bigcup_{w\geq w_{\underline{F}}} \Spec R_{\underline{{\mathcal{F}}},w}$ and $\Spec S_{\underline{\mathcal{F}}}=\bigcup_{w\geq w_{\underline{F}}} \Spec S_{\underline{\mathcal{F}},w}$ as topological spaces. Consequently we have an equality of sets
\begin{align*} \Spec R_{\underline{\mathcal{F}}}\times_{\Spec S_\infty}\Spec\mathcal{O}_L=&\bigcup_{w\geq w_{\underline{\mathcal{F}}}}(\Spec R_{\underline{\mathcal{F}},w}\times_{\Spec S_\infty}\Spec \mathcal{O}_L)\\
&=\bigcup_{w\geq w_{\underline{\mathcal{F}}}}(\Spec R_{\underline{\mathcal{F}},w}\times_{\Spec R}\Spec S) \\
&=\Spec S
\end{align*}
Indeed $\Spec S_{\underline{\mathcal{F}},w}=\Spec R_{\underline{\mathcal{F}},w}\times_{\Spec R} \Spec S$ by Proposition \ref{propcomparecomparelocalrings}. 
Hence $(\ref{intersectionasschemes})$ is true on the level of topological spaces. As $\Spec S_{\underline{\mathcal{F}}}$ is reduced it remains to show that $\Spec R_{\underline{\mathcal{F}}}\times_{\Spec S_\infty} \Spec \mathcal{O}_L$ is reduced as well. 

We have
\[\dim (\Spec R_{\underline{\mathcal{F}}}\times_{\Spec S_\infty} \Spec \mathcal{O}_L)=\dim \Spec S_{\underline{\mathcal{F}}}=\dim  \Spec R_{\underline{\mathcal{F}}}-q\]
and consequently $\Spec R_{\underline{\mathcal{F}}}\times_{\Spec S_\infty} \Spec \mathcal{O}_L$ is Cohen-Macaulay as $\Spec R_{\underline{\mathcal{F}}}$ is (see e.g.~\cite[Cor.~16.5.6]{EGAIV1}). \\
By \cite[Proposition 5.8.5]{EGAIV2}, it remains to prove that $\Spec R_{\underline{\mathcal{F}}}\times_{\Spec S_\infty} \Spec \mathcal{O}_L$ is generically reduced. 
Let $\eta\in \Spec R_{\underline{\mathcal{F}}}\times_{\Spec S_\infty} \Spec \mathcal{O}_L$ be the generic point of some irreducible component and write $\mathfrak{q}$ for the corresponding prime ideal of $R$. Then $\eta$ is the generic point of an irreducible component in $\Spec S_{\underline{\mathcal{F}},w}$ for some $w$ and
it follows that 
\[(\mathcal{O}_L\otimes_{S_\infty} R_{\underline{\mathcal{F}}})_{\mathfrak{q}}=\mathcal{O}_L\otimes_{S_\infty} (R_{\underline{\mathcal{F}}})_{\mathfrak{q}}=\mathcal{O}_L\otimes_{S_\infty} (R_{\underline{\mathcal{F}},w})_{\mathfrak{q}}=(S_{\underline{\mathcal{F}},w})_{\mathfrak{q}},\]
where the second equality is a consequence of Lemma $\ref{lemmagenericpoints}$ (iii) and the last equality is Proposition \ref{propcomparecomparelocalrings}.
As $(S_{\underline{\mathcal{F}},w})_{\mathfrak{q}}$ is reduced so is $(\mathcal{O}_L\otimes_{S_\infty} R_{\underline{\mathcal{F}}})_{\mathfrak{q}}$ and the claim follows.
\end{proof}

Let us write $R_{\cris}$ for the quotient of $R$ such that for all $v\in S_p$ and any morphism $R\rightarrow A$ (for some finite dimensional $L$-algebra $A$) the $\mathcal{G}_{E_{\vtilde}}$ representation on $A^n$ induced by $R_{\overline{\rho}_v}\rightarrow R\rightarrow A$ is crystalline. This quotient exists by the main result of \cite{Kisindef}.

If $A$ is a complete noetherian local ring, we write $t_A$ for the tangent space of the functor $\Spf A$, ie $t_A\coloneqq T\Spf A$.

\begin{lemm}\label{lemmaSelmergroup}
\begin{enumerate}[(i)]
\item The local ring $R_{\cris}$ is formally smooth of dimension $q$.
\item The tangent spaces of $R_{\cris}$ and $S$ intersect trivially inside $t_R$, i.e.~
 \[t_{R_{\cris}}\cap t_S=0.\] 
\item There is an inclusion $t_{R_{\cris}}\subset t_{R_{\underline{\mathcal{F}}}}$.
\end{enumerate}
\end{lemm}
\begin{proof}
\noindent (i) This follows from the smoothness of $\mathcal{X}_{\overline{\rho}^p}$ at the image of $\rho$ (see above) and the fact that the generic fiber of a crystalline deformation rings is smooth of dimension $n^2+[F_v:\Q_p]\tfrac{n(n-1)}{2}$ by \cite{Kisindef} and the definition of $q$.\\
\noindent (ii) With the notation introduced here this is the statement of \cite[Theorem A.1]{Allen1}.\\
\noindent (iii) This is a direct consequence of $\Spec R_{\cris}\subset \Spec R_{\underline{\mathcal{F}},w_0}\subset \Spec R_{\underline{\mathcal{F}}}$, which follows from section \ref{reminderqtriangulinecomponents}.
\end{proof}
\begin{coro}\label{directsumdecompo}
There is a direct sum decomposition 
\[t_{R_{\cris}}\oplus t_{S_{\underline{\mathcal{F}}}}=t_{R_{\underline{\mathcal{F}}}}\]
of subspaces of $t_R$. 
\end{coro}
\begin{proof}
As $\Spec S_{\underline{\mathcal{F}}}\subset \Spec S$ we have $t_{S_{\underline{\mathcal{F}}}}\subset t_S$ and hence $t_{S_{\underline{\mathcal{F}}}}\cap t_{R_{\cris}}=0$ by Lemma \ref{lemmaSelmergroup} (ii). 
Moreover $t_{R_{\cris}}\subset t_{R_{\underline{\mathcal{F}}}}$ and $t_{R_{\cris}}$ has dimension $q$. 
The claim now follows from the fact that $\Spec S_{\underline{\mathcal{F}}}\subset \Spec R_{\underline{\mathcal{F}}}$ is cut out by $q$ equations by Lemma \ref{lemmaqequations} and hence
\[\codim(t_{S_{\underline{\mathcal{F}}}},t_{R_{\underline{\mathcal{F}}}})\leq q.\]
\end{proof}

\begin{coro}\label{globaltangentspacesurjection}
The canonical map of tangent spaces
\[\bigoplus_{\underline{\mathcal{F}}}t_{S_{\underline{\mathcal{F}}}}\longrightarrow t_S\]
is a surjection. Here the sum is taken over all tuples $\underline{\mathcal{F}}=(\mathcal{F}_v)_{v\in S_p}$ of Frobenius stable flags $\mathcal{F}_v$ of $D_{\cris}(\rho_v)$.
\end{coro}
\begin{proof}
Consider the commutative diagram
\[\begin{tikzcd}
\bigoplus_{\underline{\mathcal{F}}} t_{R_{\underline{\mathcal{F}}}} \ar[r, twoheadrightarrow] \ar[d, twoheadrightarrow] & t_R \ar[d, twoheadrightarrow] \\
\bigoplus_{\underline{\mathcal{F}}} t_{R_{\underline{\mathcal{F}}}}/t_{R_{\cris}} \ar[r] & t_R/t_{R_{\cris}}  \\
\bigoplus_{\underline{\mathcal{F}}} t_{S_{\underline{\mathcal{F}}}} \ar[r, "\gamma"] \ar[u,"\alpha"] & t_S. \ar[u, "\beta"] 
\end{tikzcd}\]
It follows from Corollary \ref{directsumdecompo} that $\alpha$ is an isomorphism, and from Lemma \ref{lemmaSelmergroup} (ii) that $\beta$ is injective. Moreover the upper horizontal arrow is surjective by Corollary \ref{mainlocalcoro}. It follows from an obvious diagram chaise that $\gamma$ is a surjection. 
\end{proof}
\begin{rema}
We point out that it is a direct consequence of the proof of Corollary \ref{globaltangentspacesurjection} that the map
\[t_S\longrightarrow t_R/t_{R_{\cris}}\]
is an isomorphism. 
\end{rema}
\subsection{Proof of Theorem \ref{globalmaintheorem}}

We now prove the main result, Theorem \ref{globalmaintheorem}. With the above preparation, the final argument just follows the original method of Gouvea-Mazur \cite{Mazur} and Chenevier \cite{Chefougere} in the case of modular forms (i.e.~in the case $n=2$), resp.~in the case $n=3$.

Let us write $\mathcal{X}_{\overline{\rho},S}^{\aut, \sm}\subset\mathcal{X}_{\overline{\rho},S}^{\aut}$ for the smooth locus which is Zariski-open and dense in $\mathcal{X}_{\overline{\rho},S}^{\aut}$. 
Let us fix an irreducible component $C$ of $\mathcal{X}_{\overline{\rho},S}^{\aut}$. We need to show that $C$ is an irreducible component of $\mathcal{X}_{\overline{\rho},S}$.
By slight abuse of notations we write $C^{\sm}=C\cap \mathcal{X}_{\overline{\rho},S}^{\aut, \sm}$. 

As by definition $\mathcal{X}_{\overline{\rho},S}^{\aut}$ is the Zariski-closure of the $(G,U)$-automorphic points in $\mathcal{X}_{\overline{\rho},S}$, it follows that $C^{\sm}$ contains a $(G,U)$-automorphic point $\rho$. By the construction preceding Proposition \ref{propertiesofeigenvar} there is a point $y=(\rho,\delta)\in Y(\overline{\rho},U^p)$ and by Proposition \ref{propertiesofeigenvar} (ii) the classical, crystalline $\varphi$-generic points accumulate at $y$. It follows that there is a classical, crystalline $\varphi$-generic point $y'=(\rho',\delta')\in Y(U^p,\overline{\rho})$ such that $\rho'\in C^{\sm}$. 
We may replace $\rho$ by $\rho'$ (and $y$ by $y'$) and hence assume that $\rho$ is  $(G,U)$-automorphic and crystalline $\varphi$-generic.
\begin{prop}\label{globalequalityoftangentspaces}
There is an equality of tangent spaces
\[T_\rho \mathcal{X}_{\overline{\rho},S}^{\aut}=T_\rho \mathcal{X}_{\overline{\rho},S}.\]
\end{prop}

\begin{proof}
The inclusion $T_\rho \mathcal{X}_{\overline{\rho},S}^{\aut}\subset T_\rho \mathcal{X}_{\overline{\rho},S}$ is obvious and we need to prove the converse inclusion.

After enlarging $L$ if necessary, we may assume that the point $\rho$ is an $L$-valued point of $\mathcal{X}_{\overline{\rho},S}$ and that $L$ contains all eigenvalues of the crystalline Frobenius on the Weil-Deligne representation $\mathrm{WD}(D_{\cris}(\rho_v))$ associated to $D_{\cris}(\rho_v)$ for all $v\in S_p$.

For each choice of a tuple $\underline{\mathcal{F}}$ of complete Frobenius stable flags $\mathcal{F}_v$ in $D_{\cris}(\rho_v)$ and each Weyl group element $w\preceq w_{\underline{\mathcal{F}}}$ we have constructed points $x_{\underline{\mathcal{F}},w}\in Y(U^p,\overline{\rho})$ that map to $\rho$ under the canonical projection $f:Y(U^p,\overline{\rho})\rightarrow \mathcal{X}_{\overline{\rho},S}$.
As $f$ is locally on the source and the target a finite morphism by Proposition \ref{propertiesofeigenvar} (v),  and as the induced map \[\hat{\mathcal{O}}_{\mathcal{X}_{\overline{\rho},S},\rho}\longrightarrow \hat{\mathcal{O}}_{Y(U^p,\overline{\rho}),x_{\underline{\mathcal{F}},w}}\] is a surjection (as a consequence of \cite[Lemma 4.3.3]{BHS3} for example) we find an open neighborhood $U$ of $\rho$ in $\mathcal{X}_{\overline{\rho}}$ and for all $w_{\underline{\mathcal{F}}}\preceq w$ open neighborhoods $V_{{\underline{\mathcal{F}},w}}\subset Y(U^p,\overline{\rho})$ of $x_{\underline{\mathcal{F}},w}$ such that the restriction of $f$ is a closed immersion $V_{{\underline{\mathcal{F}},w}}\hookrightarrow U$.
As the classical points are Zariski-dense in $V_{\underline{\mathcal{F}},w}$ by Proposition \ref{propertiesofeigenvar} (ii), we find that $\bigcup_w V_{{\underline{\mathcal{F}},w}}\subset U\cap \mathcal{X}_{\overline{\rho},S}^{\aut}$.

The formation of scheme-theoretic images commutes with flat base change, hence in particular with passing to the complete local ring at $\rho$. It follows (using the notation from subsection \ref{subsectionglobaltangentspace}) that 
\[\Spec S_{\underline{\mathcal{F}}}\subset \Spec \hat{\mathcal{O}}_{\mathcal{X}_{\overline{\rho},S}^{\aut},\rho} \subset \Spec \hat{\mathcal{O}}_{\mathcal{X}_{\overline{\rho},S},\rho}=\Spec S.\]
We deduce that \[t_{S_{\underline{\mathcal{F}}}}\subset t_{\hat{\mathcal{O}}_{\mathcal{X}_{\overline{\rho},S},\rho}}=T_\rho \mathcal{X}_{\overline{\rho}}^{\aut}.\] 
As this conclusion holds true for each choice of $\underline{\mathcal{F}}$, Corollary \ref{globaltangentspacesurjection} implies the claimed inclusion $T_\rho\mathcal{X}_{\overline{\rho}}=t_S\subset T_\rho \mathcal{X}_{\overline{\rho}}^{\aut}$.
\end{proof}

We can conclude the proof of Theorem \ref{globalmaintheorem}.

\begin{proof}[Proof of Theorem \ref{globalmaintheorem}]
Given a rigid analytic space $Z$ and a point $z\in Z$ we write $\dim_zZ$ for the dimension of $Z$ at the point $z$, i.e.~for the dimension of the local ring $\mathcal{O}_{Z,z}$ of $Z$ at $z$. 

Assume in the first place that the group $U^p$ is sufficiently small to satisfy \eqref{net}.
We then have a chain of inequalities
\[\dim_\rho \mathcal{X}_{\overline{\rho},S}^{\aut}=\dim_\rho C\leq \dim_\rho \mathcal{X}_{\overline{\rho},S}\leq \dim T_\rho \mathcal{X}_{\overline{\rho},S}=\dim T_\rho \mathcal{X}_{\overline{\rho},S}^{\aut}=\dim_\rho \mathcal{X}_{\overline{\rho},S}^{\aut},\]
as $\rho$ is (by assumption) a smooth point of $\mathcal{X}_{\overline{\rho},S}^{\aut}$. Here the equality $\dim T_\rho \mathcal{X}_{\overline{\rho},S}=\dim T_\rho \mathcal{X}_{\overline{\rho},S}^{\aut}$ is Proposition \ref{globalequalityoftangentspaces}.
It follows that equality holds and hence the (necessarily unique) irreducible component $C$ of $\mathcal{X}_{\overline{\rho},S}^{\aut}$ containing $\rho$ is an irreducible component of $\mathcal{X}_{\overline{\rho},S}$.

Now assume that $U^p$ does not necessary satisfy \eqref{net}. Then we can find a place $v_1\notin S$ of $F$ such that $v_1$ is split in $E$. Let $V_{v_1}\subset G(F_{v_1})$ sufficiently small so that the group $V^p\coloneqq V_{v_1}\times\prod_{v\nmid p, v\neq v_1}U_v$ satisfies \eqref{net} and let $S'=S\cup\{v_1\}$. We have a closed immersion $\mathcal{X}_{\overline{\rho},S}\subset \mathcal{X}_{\overline{\rho},S'}$ and it follows from local-global compatibility theorems that a point of $\mathcal{X}_{\overline{\rho},S}$  is $(G,V)$-automorphic if and only if it is $(G,U)$-automorphic. Let $x$ be a $(G,U)$-automorphic point and let $Z$ be an irreducible component of $\mathcal{X}_{\overline{\rho},S}$ containing $x$. Let $Z'$ be some irreducible component of $\mathcal{X}_{\overline{\rho},S'}$ containing $Z'$. As $x$ is a $(G,V)$-automorphic point, it is a smooth point of $Z'$ and $\dim Z'=[F:\Q]\tfrac{n(n+1)}{2}$ (see \cite[Thm.~C]{Allen1}). On the other hand, we have $\dim Z\geq [F:\Q]\tfrac{n(n+1)}{2}$ so that we have $Z=Z'$. We have proved that $(G,V)$-automorphic points are Zariski-dense in $Z'=Z$. As these points are also $(G,U)$-automorphic we can conclude that $(G,U)$-automorphic points are Zariski-dense in $Z$.
\end{proof}
\subsection{Proof of Theorem \ref{maintheo}}
We finally turn to the proof of the main theorem as stated in the introduction. Thie result follows from Theorem \ref{globalmaintheorem} using base change results for unitary groups.

\begin{proof}[Proof of Theorem \ref{maintheo}]
Let $G$ be a unitary group over $F^+$ which is an outer form of $\GL_{n,\overline{F}}$, which is quasisplit at every finite place and such that $G(F^+\otimes_{\Q}\R)$ is compact. the existence of such a unitary group follows for example from the results of \cite[\S2]{Clozelautoduales}. By assumption there exists some regular cohomological cuspidal automorphic representation $\pi$ such that $\overline{\rho}\otimes_{\mathbf{F}}\overline{\Fp}\simeq\overline{\rho}_\pi$. It follows from \cite[Thm.~5.4]{Labesse} that the representation $\pi$ is the weak base change of some automorphic, automatically cuspidal, representation $\sigma$ of $G$. So that $\overline{\rho}\otimes_{\mathbf{F}}\overline{\Fp}\simeq\overline{\rho}_\sigma$. Let $U^p$ be some compact open subgroup of $G(\mathbb{A}^{p,\infty})$ such that $\sigma^{U^p}\neq0$. The representation $\pi$ is unramified at finite places $v\notin S$. Consequently it follows from \cite[Thm.~5.9]{Labesse} that we can choose the group $U^p$ spherical at places not in $S$ and that $U_v$ is spherical for all $v|p$. Then the representation $\overline{\rho}$ is $(G,U)$-automorphic. Consequently we can apply Theorem \ref{globalmaintheorem} to conclude that the Zariski closure of the $(G,U)$-automorphic points in $\mathcal{X}_{\overline{\rho},S}$ is a union of irreducible components. However it follows from Cor.~5.3, Thm.~5.4 and Thm.~5.9 in \cite{Labesse} that a point of $\mathcal{X}_{\overline{\rho},S}$ is automorphic if and only if it is $(G,U^p)$-automorphic for some $U^p$ as above. This concludes the proof.
\end{proof}

\subsection{Remarks on the existence of enough automorphic points}
We end by discussing that the main theorem conjecturally should imply density of automorphic points in $\mathcal{X}_{\overline{\rho},S}$.
Let us write $\mathcal{X}_{\overline{\rho},S}=\bigcup C_i$ for the decomposition into irreducible components. Then, obviously Theorem \ref{globalmaintheorem} implies that the $(G,U)$-automorphic points are Zariski-dense in $\mathcal{X}_{\overline{\rho},S}$, if $C_i\backslash \bigcup_{j\neq i}C_j$ contains a $(G,U)$-automorphic point for each $i$.

A result like this is the main result of Allen \cite{Allen2}. As loc.~cit.~is formulated for Galois representations associated to automorphic representations of $\GL_n$ rather than for forms on a unitary group, we repeat the argument for the convenience of the reader. The argument however is taken from \cite{Allen2}.

Given a set of Hodge-Tate weights $\mathbf{k}_v$ we write $R_{\overline{\rho}_v}^{\mathbf{k}_v-\cris}$ for the quotient of $R_{\overline{\rho}_v}$ parametrizing crystalline deformations with Hodge-Tate weights $\mathbf{k}_v$. For $\mathbf{k}=(\mathbf{k}_v)_{v\in S_p}$ we write $R_{\overline{\rho}_p}^{\mathbf{k}-\cris}$ for the completed tensor product of the local rings $R_{\overline{\rho}_v}^{\mathbf{k}_v-\cris}$.
\begin{theo}
Assume that
\begin{enumerate}
\item[(a)] the representation $\overline{\rho}$ is adequate 
\item[(b)] the group $H^0(\mathcal{G}_{E_v}, \mathrm{ad}^0(\overline{\rho}))$ vanishes for each $v\in S_p$ 
\item[(c)] there exists a lift $\rho$ of $\overline{\rho}$ that is crystalline of Hodge-Tate weight $\mathbf{k}$ and for each irreducible component $\mathcal{X}^p$ of $\mathcal{X}_{\overline{\rho}^p}$ all irreducible components of $(\Spf R_{\overline{\rho}_p}^{\mathbf{k}-\cris})^{\rig}$ are $\mathcal{X}^p$-automorphic (compare \cite[Conjecture 3.25]{BHS1}). 
\end{enumerate} 
Then each irreducible component of $\mathcal{X}_{\overline{\rho}}$ contains a $(G,U)$-automorphic point in its interior.
\end{theo}
\begin{rema}
The assumption (c) is true for example, if $\overline{\rho}$ has a potentially diagonalizable lift.
\end{rema}
\begin{proof}
Note that the points of $\mathcal{X}_{\overline{\rho}}$ are in bijection to those of $\Spec R_{\overline{\rho}}[1/p]$ and the the irreducible components of $\mathcal{X}_{\overline{\rho}}$ are in bijection to those of $\Spec R_{\overline{\rho}}[1/p]$.

The assumption $H^0(\mathcal{G}_{E_v}, \mathrm{ad}(\overline{\rho}))=0$ implies that the local deformation rings $R_{\overline{\rho}_v}$ are formally smooth. \\
We consider the patching set-up as above. Then it follows from \cite[Lemma 1.1.2]{Allen2} that the fiber product 
\begin{equation}\label{globalcrystallinelocus}
\big(\Spec \widehat{\bigotimes}_{v\in S_p}R_{\overline{\rho}_v}^{\mathbf{k}_v-\cris}\big)\times_{R_{\overline{\rho}_p}} \Spec R_{\overline{\rho}}[1/p]
\end{equation}
meets each irreducible component of $\Spec R_{\overline{\rho}}[1/p]$.
By the automorphy lifting conjecture, assumption (c), all points in the intersection $(\ref{globalcrystallinelocus})$ are $(G,U)$-automorphic. 

It is left to show that such a point $\rho$ lies on a unique irreducible component. The subspace $\mathcal{X}_{\overline{\rho}}\subset \mathcal{X}_{\infty}$ is cut out by $q$ equations and $\mathcal{X}_\infty$ is formally smooth at $\rho$. It follows that each irreducible component $C_i$ of $\mathcal{X}_{\overline{\rho}}$ has dimension larger or equal to $\dim \mathcal{X}_\infty-q$.
On the other hand, as in the proof of Corollary \ref{directsumdecompo}, we deduce from Lemma \ref{lemmaSelmergroup} (ii) that
\[\codim (t_\rho \mathcal{X}_{\overline{\rho}},t_\rho\mathcal{X}_\infty)\geq q.\]
It follows that $\rho$ is a formally smooth point of $\mathcal{X}_{\overline{\rho}}$.
\end{proof}

\end{document}